\documentclass[11pt,oneside,reqno]{article}
\usepackage[margin=1.3in]{geometry}

\usepackage{amsmath,amsthm,amssymb,color}
\usepackage[authoryear]{natbib}
\usepackage{booktabs}

\RequirePackage{amsthm,amsmath,amsfonts,amssymb,mathrsfs,dsfont,mathtools,thmtools}
\RequirePackage{natbib}
\RequirePackage[colorlinks,citecolor=blue,urlcolor=blue]{hyperref}
\usepackage[capitalize,nameinlink]{cleveref}
\RequirePackage{graphicx}

\crefname{equation}{}{}
\crefformat{equation}{\textup{(#2#1#3)}}
\crefrangeformat{equation}{\textup{(#3#1#4)--(#5#2#6)}}
\crefdefaultlabelformat{#2\textup{#1}#3}
\crefname{lemma}{Lemma}{Lemmas}
\crefname{page}{p.}{pp.}

\usepackage{bbm}
\usepackage{mathrsfs}
\usepackage{graphicx}
\usepackage{tikz}
\usepackage{enumitem}
\usetikzlibrary{arrows, automata}
\usetikzlibrary{calc}
\usetikzlibrary{positioning}

\numberwithin{equation}{section}
\allowdisplaybreaks[4]

\theoremstyle{plain}
\newtheorem{theorem}{Theorem}[section]
\newtheorem{proposition}{Proposition}[section]
\newtheorem{lemma}{Lemma}[section]

\theoremstyle{definition}

\newtheorem{remark}{Remark}[section]

\makeatletter

\newcount\minute
\newcount\hour
\newcount\hourMins
\def\now{%
\minute=\time%
\hour=\time \divide \hour by 60%
\hourMins=\hour \multiply\hourMins by 60%
\advance\minute by -\hourMins%
\zeroPadTwo{\the\hour}:\zeroPadTwo{\the\minute}%
}
\def\zeroPadTwo#1{\ifnum #1<10 0\fi#1}

\renewcommand{\cite}{\citet}

\def\^#1{\ifmmode {\mathaccent"705E #1} \else {\accent94 #1} \fi}
\def\~#1{\ifmmode {\mathaccent"707E #1} \else {\accent"7E #1} \fi}

\def\*#1{#1^\ast}
\edef\-#1{\noexpand\ifmmode {\noexpand\bar{#1}} \noexpand\else \-#1\noexpand\fi}
\def\>#1{\vec{#1}}
\def\.#1{\dot{#1}}
\def\wh#1{\widehat{#1}}
\def\wt#1{\widetilde{#1}}

\def\atop{\@@atop}
\def\*#1{\mathscr{#1}}

\renewcommand{\leq}{\leqslant}
\renewcommand{\geq}{\geqslant}

\newcommand{\eq}{\eqref}

\newcommand{\diag}{{\mathop{\mathrm{diag}}}}

\newcommand{\IE}{\mathbbm{E}}
\newcommand{\IP}{\mathbbm{P}}

\newcommand{\Cov}{\mathop{\mathrm{Cov}}}

\newcommand{\Vol}{\mathop{\mathrm{Vol}}}

\newcommand{\sgn}{\mathop{\mathrm{sgn}}}

\newcommand{\IR}{\mathbb{R}}

\def\be#1{\begin{equation*}#1\end{equation*}}
\def\ben#1{\begin{equation}#1\end{equation}}
\def\bes#1{\begin{equation*}\begin{split}#1\end{split}\end{equation*}}
\def\besn#1{\begin{equation}\begin{split}#1\end{split}\end{equation}}

\def\ba#1{\begin{align*}#1\end{align*}}
\def\ban#1{\begin{align}#1\end{align}}

\def\norm#1{\Vert#1\Vert}

\def\abs#1{\vert#1\vert}

\def\beqn#1\eeqn{\begin{align}#1\end{align}}
\def\beq#1\eeq{\begin{align*}#1\end{align*}}

\usepackage{graphicx}
\usepackage{latexsym}
\usepackage{amsmath,amsthm,amssymb,amscd}
\usepackage{epsf,amsmath}

\def\E{{\IE}}
\def\P{{\IP}}

\renewcommand\section{\@startsection {section}{1}{\z@}%
{-3.5ex \@plus -1ex \@minus -.2ex}%
{1.3ex \@plus.2ex}%
{\center\small\sc\mathversion{bold}}}

\def\subsection#1{\@startsection {subsection}{2}{0pt}%
{-3.5ex \@plus -1ex \@minus -.2ex}%
{1ex \@plus.2ex}%
{\bf\mathversion{bold}}{#1}}

\def\subsubsection#1{\@startsection{subsubsection}{3}{0pt}%
{\medskipamount}%
{-10pt}%
{\normalsize\itshape}{\kern-2.2ex. #1.}}

\def\blfootnote{\xdef\@thefnmark{}\@footnotetext}

\makeatother

\begin{document}

\title{Cram\'er-type Moderate Deviation for Quadratic Forms with a Fast Rate}
\author{Xiao Fang$^*$, Song-Hao Liu$^{\dagger}$, Qi-Man Shao$^{*,\dagger}$}
\date{\it {\small The Chinese University of Hong Kong$^*$, Southern University of Science and Technology$^\dagger$}} 
\maketitle

\noindent{\bf Abstract:} 
Let $X_1,\dots, X_n$ be independent and identically distributed random vectors in $\mathbb{R}^d$. Suppose $\mathbbm{E} X_1=0$, $\Cov(X_1)=I_d$, where $I_d$ is the $d\times d$ identity matrix. Suppose further that there exist positive constants $t_0$ and $c_0$ such that $\mathbbm{E} e^{t_0|X_1|}\leq c_0<\infty$, where $|\cdot|$ denotes the Euclidean norm. Let $W=\frac{1}{\sqrt{n}}\sum_{i=1}^n X_i$ and let $Z$ be a $d$-dimensional standard normal random vector. Let $Q$ be a $d\times d$ symmetric positive definite matrix whose largest eigenvalue is 1.
We prove that for $0\leq x\leq \varepsilon n^{1/6}$,
\begin{equation*}
\left| \frac{\mathbbm{P}(|Q^{1/2}W|>x)}{\mathbbm{P}(|Q^{1/2}Z|>x)}-1    \right|\leq
C \left( \frac{1+x^5}{\det{(Q^{1/2})}n}+\frac{x^6}{n}\right)  \quad \text{for}\ d\geq 5
\end{equation*}
and
\begin{equation*}
\left| \frac{\mathbbm{P}(|Q^{1/2}W|>x)}{\mathbbm{P}(|Q^{1/2}Z|>x)}-1    \right|\leq
C \left( \frac{1+x^3}{\det{(Q^{1/2})}n^{\frac{d}{d+1}}}+\frac{x^6}{n}\right) \quad \text{for}\ 1\leq d\leq 4,
\end{equation*}
where $\varepsilon$ and $C$ are positive constants depending only on $d, t_0$, and $c_0$. This is a first extension of Cram\'er-type moderate deviation to the multivariate setting with a faster convergence rate than $1/\sqrt{n}$. The range of $x=o(n^{1/6})$ for the relative error to vanish and the dimension requirement $d\geq 5$ for the $1/n$ rate are both optimal. We prove our result using a new change of measure, a two-term Edgeworth expansion for the changed measure, and cancellation by symmetry for terms of the order $1/\sqrt{n}$.





\medskip

\noindent{\bf AMS 2010 subject classification: }  
60F05, 60F10, 62E17

\noindent{\bf Keywords and phrases:}  
Asymptotic expansion, central limit theorem, change of measure, quadratic forms, moderate deviations.  

\section{Introduction and Main Result}

Let $X_1,\dots, X_n$ be independent and identically distributed (i.i.d.) real-valued random variables with $\IE X_1=0, \IE X_1^2=1, \IE|X_1|^3<\infty$. Let $W=\frac{1}{\sqrt{n}}\sum_{i=1}^n X_i$.
The well-known Berry--Esseen bound (\cite{berry1941accuracy}, \cite{esseen1942liapunoff}) states that
\ben{\label{f1}
\sup_{x\in \mathbb{R}}|\P(W> x)-(1-\Phi(x))|\leq \frac{\IE|X_1|^3}{\sqrt{n}},
}
where $\Phi(\cdot)$ is the standard normal distribution function.
The rate $1/\sqrt{n}$ is optimal given that the distribution function of $W$ can have jumps of size $1/\sqrt{n}$, e.g., when $X_1=\pm 1$ with probability 1/2, while $\Phi(\cdot)$ is continuous.

\cite{esseen1945fourier} first discovered an improved convergence rate in the multivariate normal approximation of sums of i.i.d.\ random vectors on centered Euclidean balls. Let $X_1,\dots, X_n$ be i.i.d.\ random vectors in $\mathbb{R}^d$, $d\geq 2$, with $\IE X_1=0,\  \Cov(X_1)=I_d,\  \IE|X_1|^4<\infty$, where $I_d$ denotes the $d\times d$ identity matrix and $|\cdot|$ denotes the Euclidean norm.
Let $W=\frac{1}{\sqrt{n}}\sum_{i=1}^n X_i$ and $Z\sim N(0, I_d)$. Then, we have (see
\cite[Chapter VII, Theorem 1]{esseen1945fourier})
\ben{\label{f2}
\sup_{x\geq 0}|\P(|W|> x)-\P(|Z|> x)|\leq \frac{C_d}{n^{\frac{d}{d+1}}} (\IE|X_1|^4)^{3/2},
}
where $C_d$ is a constant depending only on $d$.
For $d\geq 5$,  \cite[Corollary 2.3]{GtzeFriedrich2014} later proved that
\ben{\label{f2.5}
\sup_{x\geq 0}|\P(|Q^{1/2}W|> x)-\P(|Q^{1/2}Z|> x)|\leq \frac{C_d}{\det(Q^{1/2}) n} \IE|X_1|^4,
}
where $Q$ is a $d\times d$ symmetric positive definite matrix whose largest eigenvalue is 1,
and
 $C_d$ is a constant depending only on $d$.
Thus, in particular, under a finite fourth moment condition, the rate of convergence for the chi-square $\chi^2_d$
approximation of the squared Euclidean norm of a sum of i.i.d.\ random vectors $|W|^2$ can be
improved to $1/n^{\frac{d}{d+1}}$ for $2\leq d\leq 4$ and to $1/n$ for $d\geq 5$. In \eq{f2.5}, both
the threshold of the dimension, namely, 5, and the $1/n$ rate are optimal (\cite{BG1997}).

By assuming in addition that the moment generating function of $X_1$ exists in a neighborhood of 0,
\cite{cramer1938} and \cite{VonBahr1967} obtained relative error bounds for the approximation in \eq{f1} and \eq{f2}, respectively.
In particular, from \cite[Theorem 3]{VonBahr1967}, along with an expansion and symmetry argument (see \cref{s6}), we have, for $0\leq x\leq \varepsilon n^{1/6}$,
\ben{\label{f3}
\left| \frac{\P(|W|>x)}{\P(|Z|>x)}-1    \right|\leq C\left( \frac{1+x}{\sqrt{n}} +\frac{x^6}{n} \right),
}
where $\varepsilon$ and $C$ are unspecified positive constants, which do not depend on $n$ and $x$. 
We refer to results such as \eq{f3} as Cram\'er-type moderate deviations. 

The range of $x$ for the relative error in \eq{f3} to vanish, namely, $x=o(n^{1/6})$, is optimal.
More precisely, let $\{X_1, X_2,\dots\}$ be a sequence of i.i.d.\ random vectors in $\mathbb{R}^d$ with zero mean, identity covariance matrix, and $\IE e^{t_0|X_1|}\leq c_0<\infty$ for some positive constants $t_0$ and $c_0$. Let $W_n=\frac{1}{\sqrt{n}}\sum_{i=1}^n X_i, n\geq 1$.
If the mixed third cumulants of $X_1$ are not all zero, then, again from \cite{VonBahr1967}, we have (see \cref{s6}), for any fixed positive constant $c$,
\ben{\label{f23}
\frac{\P(|W_n|>cn^{1/6})}{\P(|Z|>cn^{1/6})}\nrightarrow 1, \ \text{as}\ n\to \infty.
}

By comparing \eq{f2} and \eq{f3}, we observe the following gap: Taking, say, $x=1$, in \eq{f3}, we obtain
\be{
|\P(|W|>1)-\P(|Z|>1)|\leq \frac{C}{\sqrt{n}},
}
which does not recover \eq{f2} for $d\geq 2$. 
Therefore, there is a gap in the rate of convergence between the Berry--Esseen bound \eq{f2} or \eq{f2.5} and the Cram\'er-type moderate deviation \eq{f3}.
This paper aims to establish a refined Cram\'er-type moderate deviation theorem with a rate of convergence matching that of the Berry--Esseen bound.


The following theorem is our main result.
\begin{theorem}\label{t1}
Let $X_1,\dots, X_n$ be i.i.d.\ random vectors in $\mathbb{R}^d$, where $d\geq 1$, and let Q be a symmetric
positive definite matrix whose largest eigenvalue is 1. Suppose $\IE  X_1=0,\
\Cov(X_1)=I_d$, and $\IE e^{t_0|X_1|}\leq c_0<\infty$ for some positive constants $t_0$ and $c_0$.
 Let $W=\frac{1}{\sqrt{n}}\sum_{i=1}^n X_i$ and $Z\sim N(0, I_d)$. 
Then, for $0\leq x\leq \varepsilon n^{1/6}$, we have
\ben{\label{f4}
    \left| \frac{\P(|Q^{1/2} W|>x)}{\P(|Q^{1/2} Z|>x)}-1    \right| \leq
    C \left( \frac{1+x^3}{\det(Q^{1/2})n^{\frac{d}{d+1}}} 1_{\{d\leq 4\}}
	+\frac{1+x^5}{\det(Q^{1/2}) n}1_{\{d\geq 5\}}+\frac{x^6}{n}\right),
}
where $\varepsilon$ and $C$ are positive constants depending only on $d, t_0$, and $c_0$.
\end{theorem}

\begin{remark}
\cref{t1} provides the first extension of Cram\'er-type moderate deviation to the multivariate setting with a faster convergence rate than $1/\sqrt{n}$. The convergence rates in \eq{f4} match those in \eq{f2} and \eq{f2.5}. 
In particular, the $1/n$ rate and the dimension requirement $d\geq 5$ for such a rate are optimal.
To prove \cref{t1}, we use a new change of measure, which may be of independent interest.
\end{remark}

\begin{remark}
We assume $\Cov(X_1)=I_d$ and $\norm{Q}_{op}=1$, where $\norm{\cdot}_{op}$ denotes the operator norm, in \cref{t1} without loss of generality. 
Suppose $\overline W=\frac{1}{\sqrt{n}}\sum_{i=1}^n \overline X_i$, where $\{\overline X_i\}_{i=1}^n$ are i.i.d., $\E \overline X_1=0$, $\Cov(\overline X_1)=\overline \Sigma$ (positive definite), $\overline Q$ is an arbitrary symmetric positive definite matrix, and $\overline x\geq 0$.
Then, $\P(|\overline Q^{1/2} \overline W|>\overline x)$ reduces to the setting in \cref{t1} with 
\be{
Q=\frac{\overline \Sigma^{1/2}\cdot \overline Q \cdot  \overline \Sigma^{1/2}}{\norm{\overline \Sigma^{1/2}\cdot \overline Q \cdot  \overline \Sigma^{1/2}}_{op}},\quad x=\frac{\overline x}{\norm{\overline \Sigma^{1/2}\cdot \overline Q \cdot  \overline \Sigma^{1/2}}_{op}^{1/2}}.
}
However, the condition becomes $\IE e^{t_0|\overline \Sigma^{-1/2}\cdot \overline X_1|}\leq c_0<\infty$, as in the Lyapunov-type bounds in the literature of multivariate normal approximations; see \cite{bentkus2005} and \cite{GtzeFriedrich2014}. 
\end{remark}


\begin{remark}
The factor $\frac{1}{\det(Q^{1/2})}$ in the bound \cref{f4} also appeared in \cite{GtzeFriedrich2014} (cf. \cref{f2.5}). Such a factor prevents the degenerate case: if the problem is essentially lower dimensional, then the $1/n$ rate may not be valid. 
\end{remark}




This paper is organized as follows:
In \cref{s2}, we present the details of our new change of measure and postpone the proofs of lemmas to \cref{s5}.
The proof of \cref{t1} is given in \cref{s4}.
We provide a complete proof of \eq{f3} and \eq{f23} in \cref{s6}.

In \cref{s2,s4,s5}, we use $\varepsilon$ and $C$ to denote positive constants depending only on
$d, t_0$ and $c_0$. They may differ in different expressions. 
We use $O(\cdot)$ to denote a quantity (which can be random) that is bounded in absolute value by the quantity in the parentheses multiplied by a constant depending only on
$d, t_0$, and $c_0$.

\section{A New Change of Measure}\label{s2}

Recall our setting: Let $X_1,\dots, X_n$ be i.i.d.\ random vectors in $\mathbb{R}^d$, where $d\geq 1$.
Suppose $\IE  X_1=0, \Cov(X_1)=I_d$, and $\IE e^{t_0|X_1|}\leq c_0<\infty$ for some positive
constants $t_0$ and $c_0$.  Let $W=\frac{1}{\sqrt{n}}\sum_{i=1}^n X_i$ and $Z\sim N(0, I_d)$.
Without loss of generality, we assume $Q$ is a diagonal matrix with diagonal entries $1=q_{1}\geq q_{2}\geq \ldots \geq
q_{d}>0$. Let $D=Q^{1/2}$. 
In fact, for any symmetric positive definite matrix
$Q$, there exists an orthogonal matrix $P$ such that  $Q=P^{T}\Lambda P,
$ where $\Lambda=\diag \{q_{1},q_{2},\ldots,q_{d}\}$. We then have $$\P\biggl(\biggl\lvert  Q^{1/2} \sum^{n}_{i=1}
X_{i}/\sqrt{n}\biggr\rvert>x\biggr) =\P\biggl(\biggl\lvert\Lambda^{1/2} \sum^{n}_{i=1}
Y_{i}/\sqrt{n}\biggr\rvert>x\biggr), $$ 
where $Y_{i}=PX_{i}$, $\E Y_{i}=0$, $\Cov(Y_{i})=I_{d}$, and $|Y_i|=|X_i|$ and the problem reduces to the special
case.
Moreover, we assume that $x> 1$ without loss of generality. Otherwise, \cref{t1} follows from \eq{f2.5} for the case $d\geq 5$ and from \cref{l11} (with $\Sigma_1=Q$ and $b=0$) for the case $1\leq d\leq 4$. We also assume that $x\leq \varepsilon n^{1/6}$ for a sufficiently small $\varepsilon>0$ as in the condition of \cref{t1}.

\paragraph{Proof Strategy.}
Roughly speaking, \cite{VonBahr1967} proved \eq{f3} by first using a \emph{local} exponential change of measure for different subsets $S_b$ of $\mathbb{R}^d$ nearby $b\in \mathbb{R}^d$, then using a normal approximation for the changed measure on each subset $S_b$, and finally combining the approximation results of all of the subsets. The $1/\sqrt{n}$ rate comes from the normal approximation step for each $S_b$.

In contrast, we use a new \emph{global} change of measure. This is motivated by \cite{Aleskevi1997}.
They considered, for the case $Q=I_d$ and each $x> 1$, an exponentially tilted $\widetilde W_A$ such that
\ben{\label{f5}
\P(\widetilde W_A=d y)=\frac{e^{h|y|^2}}{\IE e^{h|W|^2}}\P(W=d y),\quad y\in \mathbb{R}^d,
}
where 
\be{
h=h(x)=1/2-1/2x^2> 0.
}
We note that if $W$ is replaced by the standard normal $Z\sim N(0, I_d)$, then $\IE  e^{h|Z|^2}=x^d$ and the exponentially tilted $\widetilde Z_A$ follows $N(0, x^2I_d)$. Therefore, $\{|\widetilde Z_A|>x\}$ becomes a typical event. Because $W$ is close to normal, we may hope that $\widetilde W_A$ is close to $N(0, x^2 I_d)$ and use this approximation to obtain the desired relative error bound as in the classical change of measure argument.
However, under the condition of \cref{t1}, $\IE e^{h|W|^2}$ may be $\infty$. In fact, even if $\IE e^{h|W|^2}$ is finite, it is typically too large for $\widetilde W_A$ to be close to $N(0, x^2I_d)$. 

Observing that $e^{h|y|^2}=\IE e^{\langle \sqrt{2h}Z, y\rangle}$ for the case $Q=I_d$, we modify \eq{f5} by considering, for the case of general diagonal matrix $Q$ and $D=Q^{1/2}$, 
\ben{\label{f7}
\P(\widetilde W= d y)=\frac{\IE e^{\langle \sqrt{2h}D Z_x, y\rangle }}{\IE e^{\langle \sqrt{2h}D Z_x, W\rangle }}\P(W=d y), \quad y\in \mathbb{R}^d,
}
where $\langle A, B\rangle$ denotes the inner product, $Z_x$ is an independent standard normal random vector restricted to the centered ball with radius $z_0$, that is,
\ben{\label{f10}
\P(Z_x=dz)=\kappa 1_{(|z|\leq z_0)} \frac{1}{(\sqrt{2\pi})^d} e^{-|z|^2/2}dz,\quad z\in \mathbb{R}^d,
}
$\kappa$ is the normalizing constant and $z_0=z_0(x)=3x$ (which will be used in \eq{eq:prop9}).
Because of the assumption of finite moment generating function, $\IE e^{\langle \sqrt{2h}D Z_x, W\rangle }$
is finite for $1<x\leq \varepsilon n^{1/6}$ for a sufficiently small $\varepsilon$. 


The rest of the proof is provided in three steps.
First, we write $\P(| D\widetilde W|>x)$ as a weighted sum of probabilities involving quadratic forms (cf. \eq{f13}).
Second, we approximate each probability using a two-term Edgeworth expansion (cf. \eq{f27}).
We quantify the error in such an approximation using a result of \cite{GtzeFriedrich2014} for the
case $d\geq 5$ (cf. \cref{l3}) and a modification of a result of \cite{esseen1945fourier} for the case $1\leq d\leq 4$ (cf. \cref{l11}).
Finally, we show that the terms of the order $1/\sqrt{n}$ in the Edgeworth expansion disappear using a symmetry argument (cf. \eq{eq:lem5_16}).

\bigskip

\bigskip

Now we begin with the formal proof.
Assume without loss of generality that $\{X_i\}_{i=1}^n$, $Z$ and $Z_x$ defined above are jointly independent.
From \eq{f7}, the characteristic function of $\wt{W}$ can be expressed as
\ben{\label{f11}
\IE e^{\langle i t, \widetilde W\rangle }=\frac{\IE  e^{\langle \sqrt{2h}D Z_x+it, W\rangle }}{\IE
e^{\langle \sqrt{2h}D Z_x, W\rangle }}.
}
When the expectation is with respect to both $Z_x$ and $W$, we compute it by first conditioning on $Z_x$. 
Let $\hat{G}(b)=\IE  e^{\langle b, X_{1}\rangle } $ for a complex vector $b\in \mathbb{C}^{d}$.
We write the characteristic function of $D \wt{W}$ (cf. \eq{f11}) as
\be{
\IE  e^{\langle it,D \wt{W}\rangle }=\frac{1}{\IE  \hat G^n \bigl(\sqrt{2h}D Z_x/\sqrt{n}\bigr) }\IE  \left[
\hat G^n \bigl(\sqrt{2h} D Z_x/\sqrt{n}\bigr) \frac{\hat G^n \bigl(\bigl(\sqrt{2h}D Z_x+iDt\bigr)/\sqrt{n}\bigr)}{\hat G^n
\bigl(\sqrt{2h}D Z_x/\sqrt{n}\bigr)} \right].
}
This implies that $D \wt W$ is a mixture (depending on the value of $Z_x$) of sums of i.i.d.\ random
vectors $\frac{1}{\sqrt{n}}\sum_{i=1}^n \wt X_i$, where each $\wt X_i$ has the characteristic
function $\frac{\hat G ((\sqrt{2h}D Z_x/\sqrt{n})+iD t)}{\hat G (\sqrt{2h}D Z_x/\sqrt{n})}$, that is,
\begin{equation}
    \begin{aligned}
        \P^{Z_{x}}(\widetilde{X}_{1}=d y)= \frac{e^{\langle \sqrt{2h}  Z_x, y\rangle /\sqrt{n}}}{\hat{G}(\sqrt{2h}
        D Z_x/\sqrt{n})}\P(D X_{1}=d y).
    \end{aligned}
    \label{eq:2_2}
\end{equation}
Hereafter, we use $\P^{Z_{x}}$ and $\IE ^{Z_{x}}$ to denote the conditional probability and expectation, respectively, given $Z_x$.
Therefore,
\ben{\label{f13}
    \P(|D \wt W|\leq a)= \frac{1}{\IE \hat G^n (\sqrt{2h}D Z_x/\sqrt{n}) } \IE \left[\hat G^n
    (\sqrt{2h}D Z_x/\sqrt{n}) \P^{Z_x} \left( \left|\frac{ \sum_{i=1}^n  \wt X_i}{\sqrt{n}} \right|
\leq a \right)  \right].
}
 We will use a two-term Edgeworth expansion to approximate $\P^{Z_x} \left( \left|\frac{ \sum_{i=1}^n \wt X_i}{\sqrt{n}} \right| \leq a \right)$. 
 To express the two-term Edgeworth expansion, let $\wt \mu_1=\wt \mu_i=\IE ^{Z_x} \wt X_i$ and rewrite
\be{
\P^{Z_x} \left( \left|\frac{ \sum_{i=1}^n \wt X_i}{\sqrt{n}} \right| \leq a \right)=\P^{Z_x}
\left(\frac{ \sum_{i=1}^n (\wt X_i-\wt \mu_i)}{\sqrt{n}}\in  B(-\sqrt{n}\wt \mu_1, a) \right),
}
where $B(b, a)$ denotes the Euclidean ball with center $b$ and radius $a$.
Denote by $\wt \Sigma$ the conditional covariance matrix of $ \wt X_1$ given $Z_x$. 
It will be shown in \cref{l4} that $\wt \Sigma$ is positive definite when $1<x\leq \varepsilon n^{1/6}$ for a sufficiently small $\varepsilon>0$.
Denote by $\phi$ the $d$-dimensional standard normal density function, 
\begin{equation}
    \begin{aligned}
\wt p(y)=\phi(\wt \Sigma^{-1/2}y)/\sqrt{\det \wt \Sigma} \quad\quad \quad \mbox{(density of $N(0,\wt \Sigma)$)}
    \end{aligned}
    \label{eq:lem4_7}
\end{equation}
and 
\begin{equation}
    \begin{aligned}
\wt p'''(y) u^3=\wt p(y)\left( 3\langle \wt \Sigma^{-1}u, u\rangle\langle \wt \Sigma^{-1}y, u\rangle-\langle \wt \Sigma^{-1}y, u\rangle^3 \right),
    \end{aligned}
    \label{eq:2_3}
\end{equation}
which is the third Frechet derivative of $p$ in direction $u$.
Let 
\ben{\label{f28}
\wt \omega(y)=\wt p(y)+\frac{1}{6\sqrt{n}} \E \wt p'''(y) (\wt X_1-\wt \mu_1)^3.
}
It can be seen from \cref{lem:4} below that, when $1<x\leq \varepsilon n^{1/6}$ for a sufficiently small $\varepsilon>0$, $\wt \omega(y)$ is absolutely integrable over $\mathbb{R}^d$.
The two-term Edgeworth expansion for $\P^{Z_x} \left( \left|\frac{\sum_{i=1}^n \wt X_i}{\sqrt{n}} \right| \leq a \right)$ is given by (cf. \cite{BR86})
\ben{\label{f27}
\int_{B(-\sqrt{n}\wt \mu_1,a)} \wt \omega(y)dy=\int_{B(0,a)} \wt \omega(y-\sqrt{n}\wt \mu_1) dy.
}
 
According to \cref{f13}--\eq{f27}, we define $\wt\Phi(\cdot)$ to be a signed measure as
\begin{equation}
    \begin{aligned}
        d \wt \Phi(y)= \frac{\IE \left[\hat G^n ( \sqrt{2h}DZ_x/\sqrt{n}) \Bigl(\wt p(y-\sqrt{n}
                   \wt \mu_{1} ) +\frac{1}{6\sqrt{n}}\IE^{Z_{x}}\Bigl\{\wt p'''(y-\sqrt{n}
                      \wt  \mu_{1})
        (\widetilde{X}_{1}-\wt \mu_{1})^3\Bigr\}  \Bigr) \right]}{\IE \hat G^n ( \sqrt{2h}DZ_x/\sqrt{n})
    }dy
    \end{aligned}
    \label{eq:2_1}
\end{equation}
and we use $\wt \Phi(B(0,a))$ to approximate $\P(|D \wt W|\leq a)$.

The the main result in this section is as follows:
\begin{proposition}\label{p1}
Under the conditions of \cref{t1},
let $\wt W$ be as in \eq{f7} and $\wt \Phi$ be as in \eq{eq:2_1}.
    There exists a positive constant $\varepsilon$ such that for $1< x\leq \varepsilon
    n^{1/6}$,
\ben{\label{f6}
    \frac{\P(|D W|>x)}{\P(|DZ|>x)}-1=O\left(\frac{x^6}{n}\right) +O(x^{2})  \sup_{a\geq
    0}\abs{\P(|D\wt{W}|\leq  a)-\wt \Phi (B(0,a))}.
}
\end{proposition}


To prove \cref{p1}, we need a few lemmas. These lemmas are proved in \cref{s5}.
The first lemma estimates $\wt \mu_1$ and $\wt \Sigma$ defined above, which depend on the value of $Z_x$. 
\begin{lemma}\label{l4}
     There exists a positive constant $\varepsilon$ such that for $1< x\leq \varepsilon
    n^{1/6}$, we have, given any $Z_{x}$,
    \begin{align}
           &\wt \mu_1=\frac{\sqrt{2h} Q Z_x}{\sqrt{n}}+\frac{1}{2n}D \IE^{Z_{x}} \{\langle \sqrt{2h}
			   DZ_x ,
X_{1}\rangle^{2}X_{1}\} + D V 	,\label{f25}\\& \nonumber
\quad\quad\quad\quad\quad\quad\quad\quad\quad\quad\quad\quad\quad \mbox{each component of the
$d$-vector $V$ is $O\left(\frac{x^3}{n^{3/2}}\right)$},\\
											  &\wt \Sigma=D\biggl(I_d+\frac{1}{\sqrt{n}}
											  \IE^{Z_{x}}\{\langle \sqrt{2h} DZ_x, X_{1}\rangle 
											  X_{1}X_{1}^{T} \}+R\biggr)D, \label{f33}
\\&\nonumber\quad\quad\quad\quad\quad\quad\quad\quad\quad\quad\quad\quad\quad\mbox{each entry of the
$d\times d$ matrix $R$ is $O\left(\frac{x^2}{n}\right)$},\\
		   &\IE ^{Z_x} \wt X_{1j} \wt X_{1k} \wt X_{1l}=\IE  (D X_{1})_{j} (D
		   X_{1})_{k}(D X_{1})_{l}+ O\left(\frac{x}{\sqrt{n}}\right)q_j^{1/2} q_k^{1/2} q_l^{1/2},\quad
		   j,k,l=1,\dots, d, \label{f26}\\
         & \E^{Z_x} |\wt X_1|^4=\E|D X_1|^4+O\left(\frac{x}{\sqrt{n}}\right),   \label{f30}\\
		 &\det \wt \Sigma=\det (Q)\biggl(1+\frac{1}{\sqrt{n}}\sum_{j=1}^d \wt \lambda_j
		 +O\left(\frac{x^2}{n}\right)\biggr)\label{eq:lem2_9},\\        
		 & \wt \Sigma^{-1}=D^{-1}\left( I_d-\frac{1}{\sqrt{n}} \IE^{Z_{x}}\{\langle \sqrt{2h} D Z_x , X_{1}\rangle 
		 X_{1}X_{1}^{T} \}+R'\right)D^{-1}, \label{f24}\\
           &\quad\quad\quad\quad\quad\quad\quad\quad\quad\quad\quad\quad\nonumber\mbox{$R'$ satisfies $\sup\limits_{X,Y\in \IR^{d}, \lvert
       X\rvert=\lvert Y\rvert=1}| X^{T}R'Y|=O\left(\frac{x^{2}}{n}\right)$},
    \end{align}
where 
\ben{\label{f29}
\wt \lambda_j=\wt \lambda_j(Z_x), 1\leq j\leq d,\  \text{denote the eigenvalues of
	$\IE^{Z_{x}}\{\langle \sqrt{2h} D Z_x, X_{1}\rangle 
X_{1}X_{1}^{T} \}$},
} $(DX_{1})_{j}$ is the $j$th component of vector $DX_{1}$,
and
$X_{1}^{T}$ denotes the transpose of $X_{1}$. Moreover, 
\begin{equation}
    \begin{aligned}
        \wt \lambda_j(Z_x)=-\wt \lambda_{j}(-Z_x).
    \end{aligned}
    \label{eq:lem2_8}
\end{equation}

\end{lemma}
  The following lemma concerns the Radon--Nikodym derivative in \cref{eq:2_1}:
  \begin{lemma}
      There exists a positive constant $\varepsilon$ such that for $1< x\leq \varepsilon
    n^{1/6}$, we have, given any $Z_{x}$,
      \begin{equation}
          \begin{aligned}
              \wt p(y-\sqrt{n}
                   \wt \mu_{1} ) +\frac{1}{6\sqrt{n}}\IE^{Z_{x}}\Bigl\{\wt p'''(y-\sqrt{n}
                       \wt \mu_{1})
                    (\widetilde{X}_{1}-\wt \mu_{1})^3\Bigr\}= H_{1}(y) + H_{2}(y),\quad y\in \mathbb{R}^d,
          \end{aligned}
          \label{eq:lem4_1}
      \end{equation}
      where  
              \begin{equation}
                  \begin{aligned}
                      H_{1}(y)&= \exp \bigl\{-h \lvert D Z_{x}\rvert^{2}+ \langle \sqrt{2h} Z_x, y\rangle 
                          \bigr\}\phi(D^{-1} y) (\det D)^{-1} \\&\times\biggl(1+B_{0}+
                      O\biggl(\frac{x^{2}}{n}\biggr)\biggr)\biggl(1+B_{1}+ O\biggl(\frac{x^{4}}{n}+
					  \frac{x^{2}\lvert D^{-1} y\rvert^{2}}{n}\biggr)\biggr)\biggl(1+B_{2}+
				  O\biggl(\frac{x^{4}}{n}+ \frac{x \lvert D^{-1} y\rvert^{3}}{n}\biggr)\biggr),
                  \end{aligned}
                  \label{eq:lem4_11}
              \end{equation}
              
$$B_{0}=-\frac{1}{2\sqrt{n}} \sum^{d}_{j=1} \wt \lambda_{j},\quad  \text{(cf. \eq{f29})}$$
  $$B_{1}=\frac{1}{2\sqrt{n}} \IE^{Z_x}\Bigl\{\langle \sqrt{2h} D Z_x, 
	  X_{1}\rangle\langle X_{1}, D^{-1} y\rangle^{2}-\langle \sqrt{2h} D Z_x, 
	  X_{1}\rangle^{2}\langle X_{1}, D^{-1} y\rangle\Bigr\} ,$$
$$B_{2}=\frac{1}{6\sqrt{n}} \IE^{Z_x}  \left\{ 3\langle   X_{1}, X_{1}\rangle\langle
	D^{-1} y-\sqrt{2h} D Z_x,  {X}_{1}\rangle-
		  \langle D^{-1}y-\sqrt{2h}D Z_x,  X_{1}\rangle^3 \right\},$$
              and
              \begin{equation}
                  \begin{aligned}
                      |H_{2}(y)|&\leq C(\det D)^{-1}\phi(D^{-1}y) \exp \biggl\{-h \lvert D Z_{x}\rvert^{2}+ \sqrt{2h}
						  \langle  Z_x, y\rangle 
						  +  C\biggl(\frac{x \lvert D^{-1} y\rvert^{2}+x^{2} \lvert
					  D^{-1}  y\rvert}{\sqrt{n}}\biggr)\biggr\}
							  \\&\quad\times\Bigl(\frac{x^{2}|D^{-1}y|^{4}}{n}+\frac{x^{4}|D^{-1}y|^{2}}{n}+\frac{x^{2}}{n}\Bigr)\Bigl(1+|B_{0}|+
                      \frac{Cx^{2}}{n}\Bigr)\biggl(1+|B_{2}|+
                  C\biggl(\frac{x^{4}}{n}+ \frac{x \lvert D^{-1} y\rvert^{3}}{n}\biggr)\biggr). 
                  \end{aligned}
                  \label{eq:lem4_14}
              \end{equation}
      \label{lem:4}
  \end{lemma}

Now we are ready to prove \cref{p1}.

\begin{proof}[Proof of \cref{p1}]
By \eq{f7} we have 
\begin{equation}
    \begin{aligned}
        \P(|D W|> x)&=\IE  e^{\langle \sqrt{2h} D Z_{x}, W\rangle } \int_{\lvert D y\rvert> x}^{ } \Bigl(
        \IE e^{\langle \sqrt{2h} D Z_{x} , y\rangle }\Bigr)^{-1} d \P( \widetilde{W}\leq y)\\
        &=  \IE  e^{\langle \sqrt{2h} D Z_{x}, W\rangle } \int_{\lvert  y\rvert> x}^{ } \Bigl(
        \IE e^{\langle \sqrt{2h}  Z_{x} , y\rangle }\Bigr)^{-1} d \P(D \widetilde{W}\leq y)\\
                     &= \IE  e^{\langle \sqrt{2h} D Z_{x}, W\rangle } \int_{\lvert  y\rvert> x}^{ } \Bigl(
						 \IE e^{\langle \sqrt{2h}  Z_{x}, y\rangle }\Bigr)^{-1} d
						 \wt\Phi(y)\\
                     &\quad+ \IE  e^{\langle \sqrt{2h}D  Z_{x}, W\rangle } \int_{\lvert  y\rvert> x}^{ } \Bigl(
			\IE e^{\langle \sqrt{2h}  Z_{x}, y\rangle }\Bigr)^{-1} d \bigl(\P(D\widetilde{W}\leq
		y)-\wt\Phi(y)\bigr)\\=:&I+II,
    \end{aligned}
    \label{eq:prop1}
\end{equation}
where $\P(D\widetilde{W}\leq y)$, $y\in \mathbb{R}^d$, denotes the multivariate distribution function of $D\wt W$.
Then according to \cref{eq:2_1}, we have 
\begin{equation}
    \begin{aligned}
        I&=   \IE \{ e^{\langle \sqrt{2h}D  Z_{x}, W\rangle }\} \int_{\lvert  y\rvert> x}^{ } \Bigl(
     \IE e^{\langle \sqrt{2h} Z_{x}, y\rangle }\Bigr)^{-1} d \wt\Phi(y)\\&= \int_{\lvert  y\rvert> x}^{ } \Bigl(
       \IE e^{\langle \sqrt{2h}  Z_{x}, y\rangle }\Bigr)^{-1} \\
																		  &\quad\quad\times \IE \left[\hat G^n (\sqrt{2h}D Z_x/\sqrt{n}) \Bigl(\wt p(y-\sqrt{n}
                    \wt \mu_{1} ) +\frac{1}{6\sqrt{n}}\IE^{Z_{x}}\Bigl\{\wt p'''(y-\sqrt{n}
                       \wt \mu_{1})
        (\widetilde{X}_{1}-\wt \mu_{1})^3\Bigr\}  \Bigr) \right]dy.
    \end{aligned}
    \label{eq:prop2}
\end{equation}
By \cref{lem:4} we have
\begin{equation}
    \begin{aligned}
        I&=   \int_{\lvert  y\rvert> x}^{ } \Bigl(
        \IE e^{\langle \sqrt{2h} Z_{x}, y\rangle }\Bigr)^{-1} \IE \left[\hat G^n (\sqrt{2h}DZ_x/\sqrt{n})
H_{1}(y) \right]dy\\&\quad +\int_{\lvert  y\rvert> x}^{ } \Bigl(
        \IE e^{\langle \sqrt{2h} Z_{x}, y\rangle }\Bigr)^{-1} \IE \left[\hat G^n (\sqrt{2h}DZ_x/\sqrt{n})
H_{2}(y) \right]dy\\&= I_{1}+I_{2}.
    \end{aligned}
    \label{eq:lem4_15}
\end{equation}
We will show below that for $1<x\leq \varepsilon n^{1/6}$ with a sufficiently small $\varepsilon>0$,
\begin{equation}
    \begin{aligned}
       \hat{G}^{n}(\sqrt{2h} D Z_x/\sqrt{n})=\Bigl(1+O\Bigl( \frac{x^{6}}{n}\Bigr)\Bigr)e^{h
       |D Z_x|^{2} }(1+B_{3}),
    \end{aligned}
    \label{eq:prop11}
\end{equation}
where 
\ben{\label{f34}
B_{3}=\frac{1}{6\sqrt{n}}\IE^{Z_x}  \langle\sqrt{2h} D Z_x , 
       X_{1}\rangle^{3}=O\left(\frac{x^3}{\sqrt{n}}\right).
}
To prove \eq{eq:prop11}, we need the following lemma, whose proof is postponed to \cref{s5}.
\begin{lemma}\label{l1}
    There exist positive constants $\varepsilon$ and $C$ such that for $|a|\leq \varepsilon
    \sqrt{n}$,   
\begin{equation}
    \begin{aligned}
      &\quad \biggl\lvert \E e^{\langle a, W\rangle }-
          \exp\biggl(\frac{|a|^2}{2}\biggr)\biggl(1+\frac{\IE \langle a,
           X_{1}\rangle^{3}}{6\sqrt{n}}\biggr)\biggr\rvert\\&\leq C\biggl(\frac{1}{n}  \lvert a\rvert^{4} +\frac{1}{n}  \lvert a\rvert^{6}\biggr) \exp
            \Bigl\{  \frac{\lvert a\rvert^{2}}{2}
                +
            \frac{C\lvert a\rvert^{3}}{\sqrt{n}}    \Bigr\}.
    \end{aligned}
    \label{eq:lem3_16}
\end{equation}
\end{lemma}
Replacing $a$ in \cref{l1} with $\sqrt{2h} D Z_x$, we obtain \cref{eq:prop11}.

We first consider $I_{2}$. By \cref{eq:lem4_14,eq:prop11} and recalling that $1<x\leq \varepsilon
n^{\frac{1}{6}}$ and $\lvert Z_{x}\rvert\leq 3x$,
\begin{equation}
    \begin{aligned}
        |I_{2}|&\leq C\int_{\lvert  y\rvert> x}^{ }\biggl(\frac{x^{2}|D^{-1}y|^{4}}{n}+\frac{x^{4}|D^{-1}y|^{2}}{n}+\frac{x^{2}}{n}\biggr) (
		\IE e^{\langle \sqrt{2h} Z_{x}, y\rangle })^{-1}\phi(D^{-1}y)
			(\det{D})^{-1}\\&\quad\quad\quad\quad\quad\times \IE \biggl[(1+B_{3})
                  \exp \biggl\{ \langle \sqrt{2h} Z_x , y\rangle 
					  +  C\biggl(\frac{x \lvert D^{-1}y\rvert^{2}+x^{2} \lvert
				  D^{-1}  y\rvert}{\sqrt{n}}\biggr)\biggr\}
                              \\&\quad\quad\quad\quad\quad\times\biggl(1+|B_{0}|+
                      C\left(\frac{x^{2}}{n}\right)\biggr)\biggl(1+|B_{2}|+
			  C\biggl(\frac{x^{4}}{n}+ \frac{x \lvert D^{-1} y\rvert^{3}}{n}\biggr)\biggr)      \biggr]dy
\\&\leq C\int_{\lvert  y\rvert> x}^{ }\frac{x^{2}|D^{-1}y|^{4}}{n} (
			\IE e^{\langle \sqrt{2h}  Z_{x} , y\rangle })^{-1}\phi(D^{-1}y)
			(\det{D})^{-1}\\&\quad\quad\quad\quad\quad\times\IE \biggl[
                  \exp \biggl\{ \langle \sqrt{2h} D Z_x , y\rangle 
					  +  C\biggl(\frac{x \lvert D^{-1}y\rvert^{2}}{\sqrt{n}}\biggr)\biggr\}
                              \biggl(1+
                                   \frac{ \lvert
						  D^{-1} y\rvert^{3}}{\sqrt{n}}\biggr)      \biggr]dy 
							  \\&\leq \frac{C}{\det D}\int_{\lvert  y\rvert> x}^{ }\frac{x^{2}|D^{-1}y|^{4}}{n} \biggl(1+
                                   \frac{ \lvert
							  D^{-1}y\rvert^{3}}{\sqrt{n}}\biggr)  
							  \exp \biggl\{-\frac{\lvert D^{-1} y\rvert^{2}}{2}+
						  C\frac{x \lvert D^{-1} y\rvert^{2}}{\sqrt{n}}\biggr\}
                                   dy,
    \end{aligned}
    \label{eq:lem5_4}
\end{equation}
where in the second inequality, we used $B_0=O\bigl(\frac{x}{\sqrt{n}}\bigr),
B_2=O\bigl(\frac{x^3+|D^{-1}y|^3}{\sqrt{n}}\bigr), B_3=O\bigl(\frac{x^3}{\sqrt{n}}\bigr)$, 
and $|D^{-1}y|>x$ if $|y|>x$.
It remains to consider the integral in \cref{eq:lem5_4}, and we will use the following lemma, which is proved in \cref{s5}:
\begin{lemma}
    For any $r\geq 2-d$ and $1<x\leq \varepsilon n^{1/6}$ for a sufficiently small $\varepsilon$, 
\begin{equation}
    \begin{aligned}
        \int_{\lvert D y\rvert> x}^{ } \lvert y\rvert^{r} 
                          \exp \biggl\{-\frac{\lvert y\rvert^{2}}{2}+
                          C\frac{x \lvert y\rvert^{2}}{\sqrt{n}}\biggr\}
                                   dy
     \leq C(r) x^{r} \IP(\lvert D Z\rvert>x),
    \end{aligned}
    \label{eq:lem5_5}
\end{equation}
where $C(r)$ denotes positive constants depending only on $r$, $d$, $t_0$, and $c_0$.

	\label{lem:8}
\end{lemma}

Combining \cref{eq:lem5_4,eq:lem5_5} (the factor $\frac{1}{\det D}$ disappears after a change of variable), we obtain 
\begin{equation}
    \begin{aligned}
        I_{2}=  O\biggl(\frac{x^{6}}{n}\biggr) \IP\left( \lvert D Z\rvert\geq x \right).
    \end{aligned}
    \label{eq:lem5_6}
\end{equation}

We now consider $I_1$ in \eq{eq:lem4_15}.
By the definition of $H_{1}(y)$ in \eq{eq:lem4_11} and \cref{eq:prop11}, we have 
\ban{
                    &\quad \hat G^n (\sqrt{2h}D Z_x/\sqrt{n})H_{1}(y)\nonumber\\&= \hat G^n
                    (\sqrt{2h}DZ_x/\sqrt{n})\exp \bigl\{-h \lvert DZ_{x}\rvert^{2}+ \langle \sqrt{2h} Z_x , y\rangle 
                      \bigr\}\phi(D^{-1}y)(\det D)^{-1} \nonumber\\&\quad\times  \left( 1+B_{0}+
                      O\Bigl(\frac{x^{2}}{n}\Bigr) \right )\biggl(1+B_{1}+ O\Bigl(\frac{x^{4}}{n}+
                      \frac{x^{2}\lvert D^{-1} y\rvert^{2}}{n}\Bigr)\biggr)\biggl(1+B_{2}+
                  O\biggl(\frac{x^{4}}{n}+ \frac{x \lvert D^{-1} y\rvert^{3}}{n}\biggr)\biggr)\nonumber\\
      &=  \biggl(1+O\Bigl(\frac{x^{6}}{n}\Bigr)\biggr)  \exp \bigl\{\langle \sqrt{2h} Z_x , y\rangle 
                      \bigr\}\phi(D^{-1} y)(\det D)^{-1} (1+B_3)\nonumber \\&\quad\times\biggl(1+B_{0}+
                      O\Bigl(\frac{x^{2}}{n}\Bigr)\biggr)\biggl(1+B_{1}+ O\Bigl(\frac{x^{4}}{n}+
                      \frac{x^{2}\lvert D^{-1} y\rvert^{2}}{n}\Bigr)\biggr)\biggl(1+B_{2}+
                  O\biggl(\frac{x^{4}}{n}+ \frac{x \lvert D^{-1} y\rvert^{3}}{n}\biggr)\biggr)
                                    \nonumber     \\&= \biggl(1+O\Bigl(\frac{x^{6}}{n}\Bigr)\biggr)\exp \bigl\{\langle \sqrt{2h}  Z_x  ,y\rangle 
                                             \bigr\}\phi(D^{-1}y) (\det D)^{-1}\nonumber\\
                                             &\quad \times   \left (1+B_{0}+B_{1}+B_{2}+B_{3}+O\left(
                                                 \frac{x^{6}}{n}+ \frac{\lvert D^{-1} y\rvert^{6}}{n}
                      \right) \right )\nonumber
                                         \\&= \exp \bigl\{\langle \sqrt{2h} Z_x, y\rangle 
                                             \bigr\}\phi(D^{-1}y)(\det D)^{-1}   \left (1+B_{0}+B_{1}+B_{2}+B_{3}+O\left(
                                                 \frac{x^{6}}{n}+ \frac{\lvert D^{-1} y\rvert^{6}}{n}
                      \right) \right ),\label{eq:lem5_7}
}
    where we used $1<x\leq \varepsilon n^{1/6}$,
    \be{
        B_0=O\biggl(\frac{x}{\sqrt{n}}\biggr),\
        B_1=O\biggl(\frac{x|D^{-1}y|^2+x^2|D^{-1}y|}{\sqrt{n}}\biggr),\
        B_2=O\biggl(\frac{x^3+|D^{-1}y|^3}{\sqrt{n}}\biggr),\ B_3=O\biggl(\frac{x^3}{\sqrt{n}}\biggr),
    }
    and straightforward simplifications for terms of order $\frac{1}{n}$.
              By \cref{eq:lem4_15} and \eq{eq:lem5_7}, 
              \begin{equation}
                  \begin{aligned}
                      I_{1}&= \int_{\lvert y\rvert> x}^{ } \phi(D^{-1}y) (\det D)^{-1}\Bigl(
                          \IE e^{\langle \sqrt{2h} Z_{x}, y\rangle }\Bigr)^{-1} \\&\quad\times \IE \left[\exp \bigl\{\langle \sqrt{2h} Z_x, y\rangle 
                          \bigr\}   \left (1+B_{0}+B_{1}+B_{2}+B_{3}+O\left(
                              \frac{x^{6}}{n}+ \frac{\lvert D^{-1} y\rvert^{6}}{n}
                  \right) \right ) \right]dy\\&= I_{11}+I_{12}+I_{13},
                  \end{aligned}
                  \label{eq:lem5_8}
              \end{equation}
    where 
    \begin{equation}
        \begin{aligned}
            I_{11}=\IP(\lvert D Z\rvert>x),
        \end{aligned}
        \label{eq:lem5_9}
    \end{equation}
    \begin{equation}
        \begin{aligned}
            I_{12}=\int_{\lvert y\rvert> x}^{ } \phi(D^{-1}y) (\det D)^{-1}\Bigl(
        \IE e^{\langle \sqrt{2h} Z_{x}, y\rangle }\Bigr)^{-1} \IE \left[\exp \bigl\{\langle \sqrt{2h} Z_x, y\rangle 
                          \bigr\}   \left (B_{0}+B_{1}+B_{2}+B_{3} \right ) \right]dy,
        \end{aligned}
        \label{eq:lem5_10}
    \end{equation}
    and 
    \begin{equation}
        \begin{aligned}
            I_{13}=\int_{\lvert y\rvert> x}^{ } \phi(D^{-1}y) (\det D)^{-1}
            O\left( \frac{x^{6}}{n}+ \frac{\lvert D^{-1} y\rvert^{6}}{n} \right)dy.  
        \end{aligned}
        \label{eq:lem5_11}
    \end{equation}
    Using \cref{lem:8}, we have
    \begin{equation}
        \begin{aligned}
            I_{13}=O  \left (\frac{x^{6}}{n} \right )\IP(\lvert D Z\rvert>x).
        \end{aligned}
        \label{eq:lem5_12}
    \end{equation}
    For $i=0,1,2,3$, let $f_{i}(y)=\IE \left[\exp \bigl\{\langle \sqrt{2h} Z_x, y\rangle  \bigr\}   B_{i}
    \right]$, and we can verify
   that 
   \begin{equation}
       \begin{aligned}
           f_{i}(-y)=-f_{i}(y). 
       \end{aligned}
       \label{eq:lem5_13}
   \end{equation}
   For example, for $f_{0}(y)$, recalling \eq{eq:lem2_8}, we have
   \begin{equation}
       \begin{aligned}
           f_{0}(-y)&= \IE \biggl[\exp \bigl\{\langle \sqrt{2h} Z_x, (-y)\rangle  \bigr\}  \biggl( -\frac{1}{2\sqrt{n}}
               \sum^{d}_{j=1} \wt \lambda_{j}(Z_{x})\biggr)
    \biggr]\\
                    &= \IE \biggl[\exp \bigl\{\langle \sqrt{2h} (-Z_x), (-y)\rangle  \bigr\}  \biggl( -\frac{1}{2\sqrt{n}}
               \sum^{d}_{j=1}\wt \lambda_{j}(-Z_{x})\biggr)
    \biggr]\\&= -\IE \biggl[\exp \bigl\{\langle \sqrt{2h} Z_x , y\rangle  \bigr\}  \biggl( -\frac{1}{2\sqrt{n}}
               \sum^{d}_{j=1}\wt \lambda_{j}(Z_{x})\biggr)
    \biggr]
    \\&= -f_{0}(y),
       \end{aligned}
       \label{eq:lem5_14}
   \end{equation}
   where the second equality holds because $Z_{x}$ has a symmetric distribution, that is, $\mathcal{L}(Z_x)=\mathcal{L}(-Z_x)$.
   Because $$\phi(-D^{-1}y)\Bigl(
       \IE e^{\langle \sqrt{2h} Z_{x}, (-y)\rangle }\Bigr)^{-1}=\phi(D^{-1}y)\Bigl(
          \IE e^{\langle \sqrt{2h} Z_{x}, y\rangle }\Bigr)^{-1},$$  we have
   \begin{equation}
       \begin{aligned}
           \int_{\lvert  y\rvert> x}^{ } \phi(D^{-1}y)\Bigl(
          \IE e^{\langle \sqrt{2h} Z_{x}, y\rangle }\Bigr)^{-1} f_{i}(y)dy=0
       \end{aligned}
       \label{eq:lem5_15}
   \end{equation}
   for $i=0,1,2,3$. By \cref{eq:lem5_15}, we have 
   \begin{equation}
       \begin{aligned}
           I_{12}=0.
       \end{aligned}
       \label{eq:lem5_16}
   \end{equation}
   Using \cref{eq:lem5_8,eq:lem5_9,eq:lem5_12,eq:lem5_16}, we have 
   \begin{equation}
       \begin{aligned}
           I_{1}= \biggl(1+O(1) \frac{x^{6}}{n}\biggr) \IP(\lvert D Z\rvert\geq x).
       \end{aligned}
       \label{eq:lem5_17}
   \end{equation}
   Combining \cref{eq:lem4_15,eq:lem5_6,eq:lem5_17}, we obtain
   \begin{equation}
       \begin{aligned}
           I=\biggl(1+O(1) \frac{x^{6}}{n}\biggr) \IP(\lvert D Z\rvert\geq x).
       \end{aligned}
       \label{eq:lem5_18}
   \end{equation}
   
Finally, we consider $II$ in \eq{eq:prop1}.
Because $Z_{x}$ is symmetric with respect to $0$, we have $ \IE e^{\langle \sqrt{2h} Z_x, y\rangle }=
m(\lvert y\rvert)$ for some function $m(\cdot): \mathbb{R}^+\to \mathbb{R}$.
 Let $e_{1}=(1,0,\ldots,0)^{T}\in \IR^{d}$, then by symmetry,
\begin{equation}
    \begin{aligned}
         m(a)&= \frac{\kappa}{(\sqrt{2\pi})^{d}}\int_{\lvert z\rvert\leq z_{0}} e^{\langle \sqrt{2h} z, e_{1}a\rangle }
        e^{-\frac{\lvert z\rvert^{2}}{2}} dz \\&= \frac{\kappa}{(\sqrt{2\pi})^{d}} e^{ha^{2}}\int_{\lvert z\rvert\leq z_{0}} 
        e^{-\frac{\lvert z-\sqrt{2h} a e_{1}\rvert^{2}}{2}} dz \\&=  \frac{\kappa}{(\sqrt{2\pi})^{d}}e^{ha^{2}}\int_{\lvert z+
        \sqrt{2h} e_{1} a\rvert\leq z_{0}} 
        e^{-\frac{\lvert z\rvert^{2}}{2}} dz\\ &= \kappa e^{ha^{2}}
        \IP(\lvert Z+ \sqrt{2h} e_{1} a\rvert\leq z_{0}), 
    \end{aligned}
    \label{eq:prop6}
\end{equation}
where $\kappa$ is defined in \cref{f10}.
From the first expression of $m(a)$ in \eq{eq:prop6}, we determine that it is increasing and $m(a)\to \infty$ as $a\to \infty$ (recall that $h=\frac{1}{2}-\frac{1}{2x^2}>0$). Then, for $II$, we have 
\begin{equation}
    \begin{aligned}
        II&=  \IE  e^{\langle \sqrt{2h}D  Z_{x}, W\rangle } \int_{a> x}^{ } (
		m(a))^{-1} d \Bigl(\P(|D\widetilde{W}|\leq a)-\wt \Phi ( B(0,a))\Bigr).
    \end{aligned}
    \label{eq:prop3}
\end{equation}
By the integration by parts formula, we have 
\begin{equation}
    \begin{aligned}
        II&=  \IE  e^{\langle \sqrt{2h} D  Z_{x} , W\rangle }  (
		m(x))^{-1}  \biggl(\P\Bigl( \widetilde{W} \in D^{-1} B(0,x)\Bigr)-\wt \Phi \Bigl( B(0,x)\Bigr)\biggr)\\
          &\quad-  \IE  e^{\langle \sqrt{2h} D Z_{x}, W\rangle } \int_{a> x}^{ }  
          \biggl(\P\Bigl( \widetilde{W} \in D^{-1} B(0,a)\Bigr)-\wt \Phi \Bigl( B(0,a)\Bigr)\biggr) d(
        m(a))^{-1}.
    \end{aligned}
    \label{eq:prop4}
\end{equation}
Furthermore, recalling that $m(a)\uparrow \infty$ as $a\uparrow \infty$,
\begin{equation}
    \begin{aligned}
        \left|II\right|&\leq \IE  e^{\langle \sqrt{2h} D  Z_{x}, W\rangle }  (
    m(x))^{-1}  \biggl\lvert \biggl(\P\Bigl( \widetilde{W} \in D^{-1} B(0,x)\Bigr)-\wt \Phi \Bigl(
	B(0,x)\Bigr)\biggr)\biggr\rvert\\
          &\quad + \IE  e^{\langle \sqrt{2h} D  Z_{x}, W\rangle } \sup_{r> 0}  \biggl\lvert \biggl(\P\Bigl( \widetilde{W} \in D^{-1} B(0,r)\Bigr)-\wt \Phi \Bigl(
	B(0,r)\Bigr)\biggr)\biggr\rvert\int_{a> x}^{ }  
           d|(
        m(a))^{-1}|\\
&\leq 2 \IE  e^{\langle \sqrt{2h}  D  Z_{x}, W\rangle }  (
    m(x))^{-1}  
          \sup_{a\geq 0}  \biggl\lvert \biggl(\P\Bigl( \widetilde{W} \in D^{-1} B(0,a)\Bigr)-\wt \Phi \Bigl(
	B(0,a)\Bigr)\biggr)\biggr\rvert.
    \end{aligned}
    \label{eq:prop5}
\end{equation}
From \eq{f10}, $h=\frac{1}{2}-\frac{1}{2x^{2}}$ and $z_0=3x$, we have
\begin{equation}
    \begin{aligned}
		\IE e^{h |D Z_{x}|^{2}}&= \frac{\kappa}{(\sqrt{2\pi})^{d}} \int_{\lvert z\rvert \leq z_{0}}^{ }
		e^{-\frac{z^{T} (I-2h Q)z}{2}} d z\\
							   &= \det(I-2hQ)^{-1/2}\frac{\kappa}{(\sqrt{2\pi})^{d}} \int_{\lvert
							   z^{T} (I-2hQ)^{-1} z\rvert \leq z_{0}^{2}}^{ }
							   e^{-\frac{z^{T} z}{2}} d z \\
							   &= \kappa \det(I-2hQ)^{-1/2} \IP\Bigl(Z^{T}(I-2h Q)^{-1}Z\leq
							   z_{0}^{2}\Bigr).
    \end{aligned}
    \label{eq:prop7}
\end{equation}
Recalling that $q_{i}$, $i=1,2,\ldots,d$ are the diagonal values of $Q$ and combining
\cref{eq:prop5,eq:prop6,eq:prop7}, we have
\begin{equation}
    \begin{aligned}
		|II|&\leq 2 \sup_{a\geq 0}|\P(|D \widetilde{W}|\leq a)-\wt \Phi (B(0,a))| \frac{\IE
			e^{\langle \sqrt{2h}D Z_x,
			W\rangle}e^{-hx^2}}{\kappa \IP(\lvert Z+ \sqrt{2h} e_{1} x\rvert\leq z_{0})}
		\\&\leq 2 \sup_{a\geq 0}|\P(|D \widetilde{W}|\leq a)-\wt \Phi (B(0,a))|\frac{\IE
			e^{\langle \sqrt{2h}D Z_x,
			W\rangle }}{ \IE e^{h |DZ_x|^{2} }} \frac{\Bigl(\prod_{i=1}^{d}(1-2h q_{i})^{-1/2}\Bigr)e^{-hx^{2}}}{\IP(\lvert Z+ \sqrt{2h} e_{1} x\rvert\leq z_{0})
        }  .
    \end{aligned}
    \label{eq:prop8}
\end{equation}
Recalling that $z_{0}=3x$, $x> 1$ and $\frac{1}{2}> h=\frac{1}{2}-\frac{1}{2x^2}> 0$, we have
\begin{equation}
    \begin{aligned}
        \frac{1}{\IP(\lvert Z+ \sqrt{2h} e_{1} x\rvert\leq z_{0})
        }\leq \P(\lvert Z\rvert\leq 2x)^{-1} = O(1).
    \end{aligned}
    \label{eq:prop9}
\end{equation}
By the definition of $h$ and by \cref{eq:prop8,eq:prop9}, 
\begin{equation}
    \begin{aligned}
        II= O(1)  \sup_{a
		\geq 0}|\P(|{D \wt W}|\leq a)-\wt \Phi ( B(0,a))|
      \frac{\IE e^{\langle \sqrt{2h} D Z_x ,
        W\rangle }}{ \IE e^{h|D Z_x|^{2} }} \Biggl(\prod_{i=1}^{d}(1-2h
	q_{i})^{-1/2}\Biggr)e^{-x^{2}/2}. 
    \end{aligned}
    \label{eq:prop10}
\end{equation}
By \cref{eq:prop10}, \cref{eq:prop11}, \cref{f34} and recalling that $1<x\leq \varepsilon
 n^{1/6}$, we have
 \begin{equation}
     \begin{aligned}
         II= O(1)  \sup_{a
			 \geq 0}|\P(|D \widetilde{W}|\leq a)-\wt \Phi ( B(0,a))| \biggl(\prod_{i=1}^{d}(1-2h
	q_{i})^{-1/2}\biggr)e^{-x^{2}/2}.
     \end{aligned}
     \label{eq:lem5_19}
 \end{equation}
 Suppose that $q_{i}$, $i=1,2,\ldots,d$ take $s$ different values, which means that 
 there exist $1=\lambda_{1}> \lambda_{2}>\ldots> \lambda_{s}>0$ and positive integers
 $v_{1},v_{2},\ldots,v_{s}$ such that 
 \begin{equation}
	 \begin{aligned}
		 q_{i}=\lambda_{j} \mbox{ for $v_{j-1}+1\leq i\leq v_{j} $ and $1\leq j\leq s$}  ,
	 \end{aligned}
	 \label{eq:lem7_10}
 \end{equation}
 where $v_{0}=0$. Recalling the definition of $h$, we then have 
 \begin{equation}
	 \begin{aligned}
		\prod_{i=1}^{d}(1-2h
	q_{i})^{-1/2} =\prod_{i=1}^{s}(1-2h
	\lambda_{i})^{-v_{i}/2}=\prod_{i=1}^{s}\biggl(1-\lambda_{i}+\frac{\lambda_{i}}{x^{2}}
			\biggr)^{-v_{i}/2}.
	 \end{aligned}
	 \label{eq:lem5_21}
 \end{equation}
 Let $p= \min \{1\leq i\leq s, (1-\lambda_{i})x^{2}/\lambda_{i}> 1\}$ (with $\min \{\emptyset\}:=s+1, \prod_{i=s+1}^{s}:=1$) and $r= \sum^{p-1}_{i=1} v_{i}\leq
 d$. We then have that \cref{eq:lem5_21} is smaller than or equal to
 \begin{equation}
	 \begin{aligned}
		 \prod_{i=1}^{p-1}\left(\frac{\lambda_{i}}{x^{2}}
			 \right)^{-v_{i}/2} \prod_{j=p}^{s}(1-\lambda_{j}
			 )^{-v_{j}/2}\leq 2^{d/2} \prod_{j=p}^{s}(1-\lambda_{j}
		)^{-v_{j}/2} x^{r},
	 \end{aligned}
	 \label{eq:lem5_22}
 \end{equation}
 where we used the fact that $1/2\leq\lambda_{i}\leq 1$ for $i\leq p-1$. 

The following lemma, proved in \cref{s5}, gives a lower bound for the tail probability of a sum of weighted chi-square
random variables. Denote by $\chi^{2}_{v}$ a chi-square random variable with $v$ degrees of
freedom.
\begin{lemma}
	Let $1=\lambda_{1}>\lambda_{2}>\ldots>\lambda_{s}> 0$ be a sequence of constants, and let
	$v_{1},v_{2},\ldots,v_{s}$ be a sequence of positive integers such that $ \sum^{s}_{i=1}
	v_{s}=d$.  Suppose $\{\chi^2_{v_1},\dots, \chi^2_{v_s}\}$ are independent.
	For any $x> 1$, we have
	\begin{equation}
		\begin{aligned}
			\P\biggl( \sum_{i=1}^{s} \lambda_{i} \chi^{2}_{v_{i}}\geq x^{2}\biggr)\geq
			C_{d} \biggl[\prod_{i=p}^{s}
		(1-\lambda_{i})^{-\frac{v_{i}}{2}}\biggr] x^{r-2} e^{-\frac{x^{2}}{2}}, 
		\end{aligned}
		\label{eq:lem7_1}
	\end{equation}
	where $C_d$ is a positive constant depending only on $d$, $p= \min \{1\leq i\leq s, (1-\lambda_{i})x^{2}/\lambda_{i}> 1\}$ (with $\min \{\emptyset\}:=s+1, \prod_{i=s+1}^{s}:=1$) and $r= \sum^{p-1}_{i=1} v_{i}\leq
	d$.
	\label{lem:7}
\end{lemma}

Using \cref{eq:lem5_19}, \cref{eq:lem5_21}, \cref{eq:lem5_22} and applying \cref{lem:7}, we have
 \begin{equation}
	 \begin{aligned}
		 II&= O(1) x^{2} \P\biggl( \sum^{s}_{i=1} \lambda_{s} \chi_{v_{i}}^{2}\geq x^{2}\biggr) \sup_{a
		 \geq 0}|\P(|D\widetilde{W}|\leq a)-\wt \Phi (B(0,a))| 
		 \\&= O(1) x^{2} \P(|D Z|\geq x) \sup_{a
		 \geq 0}|\P(|D\widetilde{ W}|\leq a)-\wt \Phi ( B(0,a))|.
	 \end{aligned}
	 \label{eq:lem5_23}
 \end{equation}
 Now, combining \cref{eq:prop1,eq:lem5_18,eq:lem5_23}, we complete the proof of \cref{p1}.



\end{proof}

\section{Proof of \texorpdfstring{\cref{t1}}{Theorem 1.1}}\label{s4}

\cref{t1} immediately follows from \cref{p1,p3} given as follows:
\begin{proposition}\label{p3}
Under the conditions of \cref{t1},
let $\wt W$ be as in \eq{f7} and $\wt \Phi$ be as in \eq{eq:2_1}.
For $d\geq 5$ and $1<x\leq \varepsilon n^{1/6}$ with a sufficiently small $\varepsilon>0$, we have
\ben{\label{f14}
\sup_{a\geq 0}\abs{\P(|D \wt{W}|\leq  a)-\wt \Phi (B(0,a))}\leq \frac{C x^3}{\det (Q^{1/2})n}.
}
For $1\leq d\leq 4$ and $1<x\leq \varepsilon n^{1/6}$ with a sufficiently small $\varepsilon>0$, we have
\ben{\label{f4.5}
	\sup_{a\geq 0}\abs{\P(|D \wt{W}|\leq  a)-\wt \Phi (B(0,a))}\leq \frac{Cx}{\det(Q^{1/2})n^{\frac{d}{d+1}}}.
}
\end{proposition}

\begin{proof}[Proof of \cref{p3}]
We first prove for the case $d\geq 5$.
We rely crucially on the following lemma:
\begin{lemma}[Corollary 2.3 of \cite{GtzeFriedrich2014}]\label{l3}
Let $Y_1, \dots, Y_n$ be i.i.d.\ random vectors in $\mathbb{R}^d$ with mean 0, positive definite covariance matrix $\Sigma$, and finite fourth moments. 
Let $\sigma^2$ denote the summation of the eigenvalues of $\Sigma$.
Let $\phi$ denote the standard normal density in $\mathbb{R}^d$, and, for $y, u\in \mathbb{R}^d$, let
\ba{
p(y)=\phi(\Sigma^{-1/2}y)/\sqrt{\det \Sigma},
}
and
\be{
p'''(y) u^3=p(y)\left( 3\langle \Sigma^{-1}u, u\rangle\langle \Sigma^{-1}y, u\rangle-\langle \Sigma^{-1}y, u\rangle^3 \right) .
}
Then, 
\besn{\label{f12}
&\sup_{a\geq 0} \left| \P\left(\frac{\sum_{i=1}^n Y_i}{\sqrt{n}}\in B(b,a) \right)- \int_{B(b,a)} \left( p(y)+\frac{1}{6\sqrt{n}} \IE p'''(y) Y_1^3 \right)  dy\right| \\
    \leq& \frac{C_d \sigma^d}{n\det(\Sigma^{1/2})} \biggl(1+ \Bigl\lvert\frac{b}{\sigma} \Bigr\rvert^3\biggr)\IE |\Sigma^{-1/2}Y_1|^4,
}
where $C_d$ is a constant depending only on $d$.
\end{lemma}


Using \cref{f13,eq:2_1}, we rewrite the target $\P(|D\wt{W}|\leq a)- \wt \Phi (B(0,a))$ as
\begin{align}
&\quad \P\Bigl(|D \wt{W}|\leq a\Bigr)-\wt \Phi \left(B(0,a)\right)\nonumber \\
&= \frac{1}{\E \hat G^n (\sqrt{2h}DZ_x/\sqrt{n}) } \E \left\{\hat G^n (\sqrt{2h}DZ_x/\sqrt{n}) \left[
\P^{Z_x} \left( \left|\frac{\sum_{i=1}^n \wt X_i}{\sqrt{n}} \right| \leq a \right)  -\int_{B(-\sqrt{n}\wt \mu_1,a)} \wt \omega(y)  dy \right]  \right\} ,
\label{f15}
\end{align}
where $\wt \omega(y)$ is defined in \eq{f28}.
We bound \cref{f15} uniformly in $a\geq 0$.
From \eq{f12} with $Y_{i}=\wt X_{i}- \wt \mu_{1}$, $b=-\sqrt{n}\wt\mu_{1}$, and
$\Sigma=\wt \Sigma$, we have
\besn{\label{f16}
&\sup_{a\geq 0} \left| \P^{Z_x} \left( \left|\frac{\sum_{i=1}^n \wt X_i}{\sqrt{n}} \right| \leq a \right)- \int_{B(-\sqrt{n}\wt \mu_1,a)} \wt \omega(y)  dy\right| \\
    \leq& \frac{C_d \wt \sigma^d}{n\det(\wt \Sigma^{1/2})} \biggl(1+\Bigl\lvert\frac{\sqrt{n} \wt \mu_1}{\wt \sigma}\Bigr\rvert^3\biggr) \E^{Z_x}|\wt \Sigma^{-1/2}(\wt X_1-\wt \mu_1)|^4,
}
where $\wt \sigma^2$ denotes the summation of the eigenvalues of $\wt \Sigma$. By
\cref{l4} and recalling that $1<x\leq \varepsilon n^{1/6}$, we have 
\begin{equation}
    \begin{aligned}
        \frac{C_d \wt \sigma^d}{n\det(\wt \Sigma^{1/2})} \biggl(1+\Bigl\lvert\frac{\sqrt{n} \wt \mu_1}{\wt \sigma}\Bigr\rvert^3\biggr) \E^{Z_x}|\wt \Sigma^{-1/2}(\wt
        X_1-\wt \mu_1)|^4 \leq   \frac{C}{n \det (Q^{1/2})}x^{3} \IE |X_{1}|^{4}\leq \frac{C
        x^{3}}{n\det (Q^{1/2})}. 
    \end{aligned}
    \label{eq:lem5_20}
\end{equation}
By \cref{f15,f16,eq:lem5_20} we complete the proof of \cref{f14}.

The result \eq{f4.5} for $1\leq d\leq 4$ is proved by the same argument as for $d\geq 5$, except
that instead of \cref{l3}, we use \cref{l11} below with $Y_{i}=\wt\Sigma^{-1/2}( \wt
X_{i}-\wt\mu_{1})$, $\Sigma_{1}=\wt \Sigma$, and
$b=-\sqrt{n} \wt
\mu_{1}$. From \eq{f33}, for $1<x\leq \varepsilon n^{1/6}$ with a sufficiently small $\varepsilon$, the largest eigenvalue of $\wt \Sigma$ is smaller than 4.
Using \cref{l11} we have 
\begin{equation}
	\begin{aligned}
	&\sup_{a\geq 0} \left| \P^{Z_x} \left( \left|\frac{\sum_{i=1}^n \wt X_i}{\sqrt{n}} \right| \leq a \right)- \int_{B(-\sqrt{n}\wt \mu_1,a)} \wt \omega(y)  dy\right| \\
    \leq& \frac{C_d }{\det(\wt \Sigma^{1/2})}  \biggl(\frac{1+|
	\sqrt{n}\wt\Sigma^{-1/2} \wt \mu_{1}|}{n^{\frac{d}{d+1}}}+\frac{|\sqrt{n}\wt\Sigma^{-1/2} \wt \mu_{1}|^{\frac{d-1}{2}}}{n}\biggr)\bigl(\IE |\wt\Sigma^{-1/2}( \wt
X_{i}-\wt\mu_{1})|^4\bigr)^{3/2}.
	\end{aligned}
	\label{eq:lem3_19}
\end{equation}
Similar to \eq{eq:lem5_20} and using $d\leq 4$,
we have
\begin{equation}
	\begin{aligned}
		&\quad \frac{C_d }{\det(\wt \Sigma^{1/2})}  \biggl(\frac{1+|
	\sqrt{n}\wt\Sigma^{-1/2} \wt \mu_{1}|}{n^{\frac{d}{d+1}}}+\frac{|\sqrt{n}\wt\Sigma^{-1/2} \wt \mu_{1}|^{\frac{d-1}{2}}}{n}\biggr)\bigl(\IE |\wt\Sigma^{-1/2}( \wt
	X_{i}-\wt\mu_{1})|^4\bigr)^{3/2}\\&\leq  \frac{C}{\det(Q^{1/2})}\biggl(\frac{1+x
}{n^{\frac{d}{d+1}}}+\frac{x^{\frac{3}{2}}}{n}\biggr)\bigl(\IE |X_{1}|^4\bigr)^{3/2}\leq \frac{Cx}{\det(Q^{1/2})n^{\frac{d}{d+1}}}.
	\end{aligned}
	\label{eq:lem3_20}
\end{equation}
From \cref{f15,eq:lem3_19,eq:lem3_20}, we complete the proof of \cref{f4.5}.
\begin{lemma}\label{l11}
Le $Y_1, \dots, Y_n$ be i.i.d.\ random vectors in $\mathbb{R}^d$ with mean 0, covariance
matrix $I_{d}$, and finite fourth moments. 
Let $\phi$ denote the standard normal density in $\mathbb{R}^d$, and, for $y, u\in \mathbb{R}^d$,
let $p(y)$ and $p'''(y) u^3$ be as defined in \cref{l3} with $\Sigma=I_{d}$. Let
$\Sigma_{1}$ be a symmetric positive definite matrix
with $\norm{\Sigma_1}_{op}\leq 4$. 
Then, for any $b\in \mathbb{R}^d$, we have 
\bes{
&\sup_{a\geq 0} \left| \P\left(\frac{\sum_{i=1}^n Y_i}{\sqrt{n}}\in  \Sigma_{1}^{-1/2} B(b,a) \right)-
\int_{y\in\Sigma_{1}^{-1/2}B(b,a)}
\left( p(y)+\frac{1}{6\sqrt{n}} \IE p'''(y) Y_1^3 \right)  dy\right| \\
	\leq& \frac{C_d}{\det\bigl(\Sigma_{1}^{1/2}\bigr)} \biggl(\frac{1+|\Sigma_{1}^{-1/2}b|}{n^{\frac{d}{d+1}}}+\frac{|\Sigma_{1}^{-1/2}b|^{\frac{d-1}{2}}}{n}\biggr)\bigl(\IE |Y_1|^4\bigr)^{3/2},
}
where $C_d$ is a constant depending only on $d$.
\end{lemma}

\begin{proof}[Proof of \cref{l11}]
In this proof, we denote by $C_d$ positive constants that depend only on $d$. They may differ in different expressions.

If $\Sigma_1=I_d$ and $b=0$, then \cref{l11} follows from \cite[Chapter VII, Theorem 1]{esseen1945fourier} by observing that $\int_{y\in B(0,a)} \E p'''(y) Y_1^3 dy=0$.
The proof for the general case is a straightforward modification (outlined below) of the proof of \cite[Chapter VII, Theorem 1]{esseen1945fourier} . 
Concerning notation, we use, e.g., (Eq. 60) to denote the equation (60) in \cite[Chapter VII]{esseen1945fourier}.
To be consistent with the notation in \cite{esseen1945fourier}, in this proof, we use the symbol $\varepsilon$ to denote a different quantity from that in the rest of the paper.
Other notations used in this proof are as follows:
$x=(x_{1},x_{2},\cdots,x_{d})^{T}$, $t=(t_{1},t_{2},\cdots,t_{d})^{T}$, $r=|t|$, and $J_{d/2}(\cdot)$ denotes the Bessel function of order $d/2$.

We first give a smoothing inequality for the noncentered ellipsoid $\Sigma_1^{-1/2}B(b,a)$ (cf. \eq{eq:lem9_9} below). 
For $\varepsilon>0$, let 
$$
Q_{\varepsilon}\left(x_{1}, x_{2}, \ldots, x_{k}\right)=\left\{\begin{array}{l}
		1 \text { for } \lvert x\rvert \leq \varepsilon \\
0 \text {  for } \lvert x\rvert>\varepsilon.
\end{array}\right.
$$
It has the following Fourier transform (cf. (Eq. 43) and (Eq. 44)):
\begin{equation}
	\begin{aligned}
	q_{\varepsilon}\left(t\right)=\int_{\mathbb{R}^{d}}^{ } e^{i\langle t,x\rangle} Q_{\varepsilon}(x) dx=\left(\frac{2 \pi a}{|t|}\right)^{k / 2} J_{k / 2}(\varepsilon
	|t|).	
	\end{aligned}
	\label{eq:lem9_3}
\end{equation}
Let $$
\wt Q_{a,b}\left(x_{1}, x_{2}, \ldots, x_{k}\right)=\left\{\begin{array}{l}
		1 \text { for } \lvert \Sigma_{1}^{1/2} x-b\rvert \leq a \\
0 \text {  for } \lvert \Sigma_{1}^{1/2} x-b\rvert>a
\end{array}\right.
$$
be the indicator of the ellipsoid $\Sigma_1^{-1/2}B(b,a)$.
From \eq{eq:lem9_3}, it has the following Fourier transform:
\begin{equation*}
	\begin{aligned}
		\wt q_{a,b}(t)=\int_{\mathbb{R}^{d}}^{ } e^{i\langle t,x\rangle} \wt Q_{a,b}(x) dx=\Biggl(
			\frac{2\pi a}{|\Sigma_{1}^{-1/2} t|}\Biggr)^{d/2} J_{d/2}(a|\Sigma_{1}^{-1/2} t|)e^{i \langle t,
		\Sigma_{1}^{-1/2}b\rangle} \frac{1}{\det(\Sigma_1^{1/2})}.
	\end{aligned}
	\label{eq:lem9_1}
\end{equation*}
Now, consider the convolution function (cf. (Eq. 45)), for $0<\varepsilon<a$,
$$M\left(x\right)=\frac{\Gamma\left(1+\frac{d}{2}\right)}{\pi^{d
/ 2} \varepsilon^{d}} \int_{\mathbb{R}^{d}} \wt Q_{a,b}\left(x_{1}-\xi_{1}, \ldots, x_{d}-\xi_{d}\right)
Q_{\varepsilon}\left(\xi_{1}, \ldots, \xi_{d}\right) d \xi_{1} \ldots d \xi_{d}.$$
 Let $A(b,a,\varepsilon,\Sigma_{1})=
 \cup_{t\in \Sigma_{1}^{-1/2} B(b,a)}B(t,\varepsilon)$ and
 $$\hat{A}(b,a,\varepsilon,\Sigma_{1})=\left(\Sigma_{1}^{-1/2} B(b,a)\right)\Big\backslash\cup_{t\in
	 (\Sigma_{1}^{-1/2} B(b,a))^{c}}B(t,\varepsilon).$$ We observe that $|M(x)| \leq 1 \text {
 for all } x$ and (cf. (Eq. 46))
 $$M\left(x_{1}, x_{2}, \ldots, x_{k}\right)=\left\{\begin{array}{l}\text { 1
		 for } x\in \hat{A}(b,a,\varepsilon,\Sigma_{1}) \\ \text { 0 for }
 x\in \left(A(b,a,\varepsilon,\Sigma_{1})\right)^{c}. \\ \end{array}\right. $$
The Fourier transform of $M, m\left(t\right)$, is (cf. (Eq. 47))
\begin{equation}
	\begin{aligned}
		m(t)= \left(
			\frac{2\pi a}{|\Sigma_{1}^{-1/2} t|}\right)^{d/2} J_{d/2}\left(a|\Sigma_{1}^{-1/2} t|\right)e^{i \langle t,
			\Sigma_{1}^{-1/2}b\rangle} 2^{d/2} \Gamma\left(1+\frac{d}{2}\right) \frac{J_{d/2}(\varepsilon
	r)}{(\varepsilon r)^{d/2}}	\frac{1}{\det(\Sigma_1^{1/2})},
	\end{aligned}
	\label{eq:lem9_2}
\end{equation}
because the Fourier transform of a convolution is equal to the product of the transforms corresponding to the functions in the convolution.

Thus, replacing $a$ by $a+\varepsilon/2$ and $\varepsilon$ by $\varepsilon/4$ in \eq{eq:lem9_2},  the function (cf. (Eq. 48))
\begin{equation*}
	\begin{aligned}
		\left(
			\frac{2\pi (a+\varepsilon/2)}{|\Sigma_{1}^{-1/2} t|}\right)^{d/2}
			J_{d/2}\left((a+\varepsilon/2)|\Sigma_{1}^{-1/2} t|\right)e^{i \langle t,
			\Sigma_{1}^{-1/2}b\rangle} 2^{d/2} \Gamma\left(1+\frac{d}{2}\right) \frac{J_{d/2}(\varepsilon
	r/4)}{(\varepsilon r/4)^{d/2}}	\frac{1}{\det(\Sigma_1^{1/2})}
	\end{aligned}
	\label{eq:lem9_5}
\end{equation*}
is the Fourier transform of a function $=$
\begin{equation}
	\begin{aligned}
		=\left\{\begin{array}{l}\text { 1
				for } x\in \Sigma_{1}^{-1/2} B(b,a) \\ \text { 0 for }
		x\in \left(\Sigma_{1}^{-1/2} B(b,a+\varepsilon) \right)^{c}.\\ \end{array}\right.
	\end{aligned}
	\label{eq:lem9_6}
\end{equation}
Similarly, the function (cf. (Eq. 49))
\begin{equation*}
	\begin{aligned}
		\left(
			\frac{2\pi (a-\varepsilon/2)}{|\Sigma_{1}^{-1/2} t|}\right)^{d/2}
			J_{d/2}\left((a-\varepsilon/2)|\Sigma_{1}^{-1/2} t|\right)e^{i \langle t,
			\Sigma_{1}^{-1/2}b\rangle} 2^{d/2} \Gamma\left(1+\frac{d}{2}\right) \frac{J_{d/2}(\varepsilon
	r/4)}{(\varepsilon r/4)^{d/2}}	\frac{1}{\det(\Sigma_1^{1/2})}
	\end{aligned}
	\label{eq:lem9_7}
\end{equation*}
is the Fourier transform of a function $=$ 
\begin{equation}
	\begin{aligned}
		=\left\{\begin{array}{l}\text { 1
				for } x\in \Sigma_{1}^{-1/2} B(b,a-\varepsilon) \\ \text { 0 for }
		x\in \left(\Sigma_{1}^{-1/2} B(b,a) \right)^{c}.\\ \end{array}\right.
	\end{aligned}
	\label{eq:lem9_8}
\end{equation}
By the well-known properties of Bessel functions (cf. (Eq. 50)):
$$
\left\{\begin{array}{l}
\left|\frac{J_{d / 2}(z)}{z^{d / 2}}\right| \leq C_d \text { for all positive } z \\
\left|J_{d / 2}(z)\right| \leq \frac{C_d}{\sqrt{z}} \text { for all positive } z,
\end{array}\right.
$$
and fact that (recall our assumption that $\norm{\Sigma_1}_{op}\leq 4$)
$$\frac{1}{|\Sigma_{1}^{-1/2}t|}\leq \frac{2}{\lvert
t\rvert}=\frac{2}{r}, $$ we have the following lemma:
\begin{lemma} (cf. Lemma 4 of \cite[Chapter VII]{esseen1945fourier})
	 Let $a$ and $\varepsilon$ be two assigned constants and $0<\varepsilon<a .$ There exists a
	 function $H\left(x, b, a, \varepsilon\right)$ such that
\begin{equation*}
	\begin{aligned}
		 H(x, b, a, \varepsilon)=\left\{\begin{array}{l}
1 \text { for } x\in \Sigma_{1}^{-1/2} B(b,a) \\
0 \text { for } x\in \left(\Sigma_{1}^{-1/2} B(b,a+\varepsilon)\right)^{c}
\end{array}, \text { and }|H(x, b, a, \varepsilon)| \leq 1\right.
	\end{aligned}
	\label{eq:lem9_26}
\end{equation*}
for all $x$.

Furthermore, the Fourier transform of $H$, $h\left(t, b, a ,\varepsilon\right)$, can be bounded by a
function depending on $t$ only through $r=\lvert t\rvert$, i.e.,
\begin{equation}
	\begin{aligned}
		|h(t, b, a, \varepsilon)| \leq \frac{C}{\det(\Sigma_1^{1/2})} \cdot \frac{a^{\frac{k-1}{2}}}{r^{\frac{k+1}{2}}},
	\end{aligned}
	\label{eq:lem9_28}
\end{equation}
\begin{equation}
	\begin{aligned}
		|h(t, b, a, \varepsilon)| \leq \frac{C}{\det(\Sigma_1^{1/2})} \cdot \frac{a^{\frac{k-1}{2}}}{\varepsilon^{\frac{k}{2}} r^{\frac{2 k+1}{2}}}.
	\end{aligned}
	\label{eq:lem9_29}
\end{equation}
There also exists a function $H(x, b, a,-\varepsilon)$ such that
\begin{equation*}
	\begin{aligned}
		 H(x, b, a, -\varepsilon)=\left\{\begin{array}{l}
1 \text { for } x\in \Sigma_{1}^{-1/2} B(b,a-\varepsilon) \\
0 \text { for } x\in \left(\Sigma_{1}^{-1/2} B(b,a)\right)^{c}
\end{array}, \text { and }|H(x, b, a, \varepsilon)| \leq 1\right.
	\end{aligned}
	\label{eq:lem9_30}
\end{equation*}
for all $x$, the Fourier transform of which, $h(t, b, a,-\varepsilon)$, satisfies the inequalities
\cref{eq:lem9_28,eq:lem9_29}.
	\label{lem:10}
\end{lemma}
Let
\be{
	\mu_n(b,a)=\P\left(\frac{\sum_{i=1}^n Y_i}{\sqrt{n}}\in \Sigma_{1}^{-1/2}B(b,a) \right)
}
and
\be{
	\psi(b,a)=\int_{\Sigma_{1}^{-1/2}B(b,a)} \left( p(y)+\frac{1}{6\sqrt{n}} \IE p'''(y) Y_1^3 \right)  dy.
}
We denote by $\Delta_{n}$ the difference of the characteristic functions of 
$$\frac{\sum_{i=1}^n Y_i}{\sqrt{n}}\quad  \text{ and } \quad   p(y)+\frac{1}{6\sqrt{n}} \IE p'''(y) Y_1^3. $$

Then, by \cref{lem:10} and the same argument as that in  \cite[p.104]{esseen1945fourier} leading to (Eq. 56), we have
\begin{equation}
	\begin{aligned}
		\left\lvert \mu_n(b,a)-\psi(b,a)  \right\rvert\leq \max\{A_{1},A_{2}\},
	\end{aligned}
	\label{eq:lem9_9}
\end{equation}
where 
\begin{equation}\label{f36}
A_{1}=|\psi(b, a+\varepsilon)-\psi(b, a)|
+\frac{1}{(2 \pi)^{k}} \int_{R_{k}}\left|\Delta_{n}(t) h(t, b, a, \varepsilon)\right| d t
\end{equation}
and 
\begin{equation*}
A_{2}=|\psi(b, a)-\psi(b, a-\varepsilon)|
+\frac{1}{(2 \pi)^{k}} \int_{R_{k}}\left|\Delta_{n}(t) h(t, b, a, -\varepsilon)\right| d t.
\end{equation*}
Similar to (Eq. 59) and (Eq. 60), we make the following
assumptions without loss of generality:
\begin{equation}
	\begin{aligned}
		1^\circ \quad a\leq 4|\Sigma_{1}^{-1/2} b|+4\log(2+n),
	\end{aligned}
	\label{eq:lem9_10}
\end{equation}
or else we choose $\varepsilon=a/8$ and proceed as in the subsequent estimations.
\begin{equation}
	\begin{aligned}
		2^\circ \quad \frac{1}{n^{\frac{d}{d+1}}}\left(\E|Y_1|^4\right)^{3/2}\leq
		\frac{1}{8},
	\end{aligned}
	\label{eq:lem9_11}
\end{equation}
or else \cref{l11} is true with a sufficiently large $C_{d}$. Choose (cf. (Eq. 61))
\begin{equation}
	\begin{aligned}
		\varepsilon=\frac{a}{n^{d/(d+1)}}\left(\E|Y_1|^4\right)^{3/2}.
	\end{aligned}
	\label{eq:lem9_17}
\end{equation}

We may confine ourselves to the estimation of $A_1$, $A_2$ being treated similarly.
To obtain an upper bound for $|\psi(b, a+\varepsilon)-\psi(b, a)|$,
we first consider 
 \begin{equation*}
    \begin{aligned}
		\int_{ y\in \Sigma_{1}^{-1/2} B(b,a+\varepsilon) \backslash \Sigma_{1}^{-1/2} B(b,a)} p(y) dy.
    \end{aligned}
    \label{eq:lem6_1}
\end{equation*}
For $a\leq  4\lvert \Sigma_{1}^{-1/2} b\rvert+ \left(4 \log(2+n)\wedge  \lvert \Sigma_{1}^{-1/2} b\rvert\right)$, from Gaussian anti-concentration inequalities (cf. \cite[Chapter 1, Section 3]{BR86}),
\begin{equation*}
    \begin{aligned}
		\int_{y\in \Sigma_{1}^{-1/2} B(b,a+\varepsilon)\backslash \Sigma_{1}^{-1/2} B(b,a)} p(y) dy\leq
		\frac{C_d
		\varepsilon}{\sigma_{min}}\leq
		\frac{C_d|\Sigma_{1}^{-1/2}b|}{\det{(\Sigma_{1}^{1/2})} n^{\frac{d}{d+1}}}\left(\IE |Y_1|^4\right)^{3/2},
    \end{aligned}
\end{equation*}
where $\sigma_{min}$ is the smallest eigenvalue of $\Sigma_{1}$, and, in the second inequality, we used $\norm{\Sigma_1}_{op}\leq 4$ and \eq{eq:lem9_17}.
If $ \lvert \Sigma_{1}^{-1/2} b\rvert\leq 4\log (n+2)$, we must also consider the case $5 \lvert \Sigma_{1}^{-1/2}b\rvert\leq
a\leq 4\lvert \Sigma_{1}^{-1/2}b\rvert+ 4 \log(2+n)$. In this situation, 
\begin{equation*}
    \begin{aligned}
     & \quad  \int_{\Sigma_{1}^{1/2}y\in B(b,a+\varepsilon)\backslash B(b,a)} p(y) dy
   \\&\leq \sup_{\Sigma_{1}^{1/2}y\in B(b,a+\varepsilon)\backslash B(b,a)}\frac{C_d
   a^{d}}{\det\left(\Sigma_{1}^{1/2}\right)n^{\frac{d}{d+1}}} \exp \{- \lvert y\rvert^{2}/2\}  \left(\IE
   |Y_1|^4\right)^{3/2},
       \end{aligned}
\end{equation*}
where we used the inequality that the volume of
$\Sigma_{1}^{-1/2}\bigl(B(b,a+\varepsilon)\backslash B(b,a)\bigr)$ is smaller than $C_d a^{d}\left(\IE
   |Y_1|^4\right)^{3/2}/(\det(\Sigma_{1}^{1/2})n^{\frac{d}{d+1}})$. In fact,
   \begin{equation*}
	   \begin{aligned}
		   \Vol \left(\Sigma_{1}^{-1/2}\bigl(B(b,a+\varepsilon)\backslash
			   B(b,a)\bigr)\right)&=\int_{y\in\Sigma_{1}^{-1/2}\bigl(B(b,a+\varepsilon)\backslash
				   B(b,a)\bigr)}^{} dy\\&= \det(\Sigma_{1}^{-1/2})\int_{y\in B(b,a+\varepsilon)\backslash
			   B(b,a)}^{} dy\\&\leq \frac{\Vol\left(B(0,a+\varepsilon)\backslash B(0,a)\right)}{\det\left(\Sigma_{1}^{1/2}\right)}
							\\&\leq \frac{C_d a^{d}\left(\IE
   |Y_1|^4\right)^{3/2}}{\det(\Sigma_{1}^{1/2})n^{\frac{d}{d+1}}}, 
	   \end{aligned}
   \end{equation*}
   where $\Vol(A)$ denotes the volume of $A\subset \mathbb{R}^{d}$ and in the last inequality, we used \eq{eq:lem9_11} and \eq{eq:lem9_17}.
   Furthermore, because $a\leq 5
\lvert y\rvert$ (which follows from $\Sigma_{1}^{1/2}y\in B(b,a+\varepsilon)\backslash B(b,a)$, $5|\Sigma_1^{-1/2}b|\leq a$, and the assumption that $\norm{\Sigma_1}_{op}\leq 4$), we have
\begin{equation}
    \begin{aligned}
       & \quad\sup_{\Sigma_{1}^{1/2}y\in B(b,a+\varepsilon)\backslash B(b,a)}\frac{C_d
   a^{d}}{\det\left(\Sigma_{1}^{1/2}\right)n^{\frac{d}{d+1}}} \exp \{- \lvert y\rvert^{2}/2\}  \left(\IE |Y_1|^4\right)^{3/2}
   \\&\leq \sup_{\Sigma_{1}^{1/2}y\in B(b,a+\varepsilon)\backslash B(b,a)}\frac{C_d \lvert
   y\rvert^{d}}{\det\left(\Sigma_{1}^{1/2}\right)n^{\frac{d}{d+1}}} \exp \{- \lvert y\rvert^{2}/2\}  \left(\IE |Y_1|^4\right)^{3/2}
   \\&\leq \frac{C_d}{\det\left(\Sigma_{1}^{1/2}\right)n^{\frac{d}{d+1}}}  \left(\IE |Y_1|^4\right)^{3/2}.
    \end{aligned}
\end{equation}
Therefore,
\ben{
	\label{lem9_13}
\int_{y\in \Sigma_{1}^{-1/2} B(b,a+\varepsilon)\backslash \Sigma_{1}^{-1/2} B(b,a)} p(y) dy\leq C_d \frac{1+|\Sigma_{1}^{-1/2}b|}{\det\left(\Sigma_{1}^{1/2}\right)n^{\frac{d}{d+1}}} \left(\E|Y_1|^4\right)^{3/2}.
}
From $|\E p'''(y) Y_1^3|\leq C_d \E|Y_1|^3 \left(|y|+\lvert y\rvert^3\right) p(y)$, by similar arguments we have
\begin{equation}
    \begin{aligned}
        \int_{y\in \Sigma_{1}^{-1/2} B(b,a+\varepsilon)\backslash \Sigma_{1}^{-1/2} B(b,a)} \frac{1}{6\sqrt{n}} |\E p'''(y) Y_1^3| dy\leq
       C_d  \frac{1+\lvert\Sigma_{1}^{-1/2} b\rvert}{\det\left(\Sigma_{1}^{1/2}\right) n^{\frac{d}{d+1}}}  \left(\IE |Y_1|^4\right)^{3/2},
    \end{aligned}
	\label{lem9_12}
\end{equation}
where we used
$\frac{\E|Y_1|^3}{\sqrt{n}}\leq \sqrt{\frac{\left(\E|Y_1|^4\right)^{3/2}}{n}}$ and \eq{eq:lem9_11}. Thus, by \cref{lem9_13,lem9_12} we have (cf. (Eq. 62))
\begin{equation}
	\begin{aligned}
		|\psi(b, a+\varepsilon)-\psi(b, a)|\leq  C_d  \frac{1+\lvert\Sigma_{1}^{-1/2}
		b\rvert}{\det\left(\Sigma_{1}^{1/2}\right) n^{\frac{d}{d+1}}}  \left(\IE
	|Y_1|^4\right)^{3/2}.
	\end{aligned}
	\label{eq:lem9_14}
\end{equation}
To bound \eq{f36}, it remains to consider (cf. (Eq. 63))
\begin{equation}
	\begin{aligned}
		I&:=\frac{1}{(2 \pi)^{k}} \int_{R_{k}}\left|\Delta_{n}(t) h(t, b, a, \varepsilon)\right| d
		t\\&=\frac{1}{(2 \pi)^{k}} \int\limits_{0 \leq r \leq \frac{\sqrt{n}}{\left(d \beta_{4}\right)^{3 /
		4}}}+\frac{1}{(2 \pi)^{k}} \int\limits_{ r>\frac{\sqrt{n}}{\left(d \beta_{4}\right)^{ 3/
4}}}=I_{1}+I_{2},
	\end{aligned}
	\label{eq:lem9_15}
\end{equation}
where $\beta_{4}=\IE \lvert Y_{1} \rvert^{4}$.
For $I_{1}$, by an argument similar to that in (Eq. 64), we have 
\begin{equation}
	\begin{aligned}
		I_{1}\leq \frac{C_d}{\det(\Sigma_{1}^{1/2})} \left( \frac{|\Sigma_{1}^{-1/2}
		b|^{\frac{d-1}{2}}}{n}+\frac{1}{n^{\frac{d}{d+1}}}\right)\left(\E|Y_1|^4\right)^{3/2} .
	\end{aligned}
	\label{eq:lem9_18}
\end{equation}
By an argument similar to that leading to (Eq. 76), we have 
\begin{equation}
	\begin{aligned}
		I_{2}\leq \frac{C_d}{n^{\frac{d}{d+1}}\det(\Sigma_{1}^{1/2})}  \left(\E|Y_1|^4\right)^{3/2}. 
	\end{aligned}
	\label{eq:lem9_20}
\end{equation}
Using \cref{eq:lem9_15,eq:lem9_18,eq:lem9_20}, we obtain
\begin{equation}
	\begin{aligned}
		I\leq \frac{C_d}{\det(\Sigma_{1}^{1/2})} \left( \frac{|\Sigma_{1}^{-1/2}
		b|^{\frac{d-1}{2}}}{n}+\frac{1}{n^{\frac{d}{d+1}}}\right)\left(\E|Y_1|^4\right)^{3/2}. 
	\end{aligned}
	\label{eq:lem9_21}
\end{equation}
Therefore, by \cref{eq:lem9_9,eq:lem9_14,eq:lem9_21}, we have 
\begin{equation}
	\begin{aligned}
		A_{1}\leq \frac{C_d}{\det\bigl(\Sigma_{1}^{1/2}\bigr)}
		\biggl(\frac{1+|\Sigma_{1}^{-1/2}b|}{n^{\frac{d}{d+1}}}+\frac{|\Sigma_{1}^{-1/2}b|^{\frac{d-1}{2}}}{n}\biggr)\bigl(\IE
		|Y_1|^4\bigr)^{3/2},
	\end{aligned}
	\label{eq:lem9_22}
\end{equation}
and thus we complete the proof of \cref{l11}.

\end{proof}

\end{proof}



\section{Proofs of Lemmas}\label{s5}

\begin{proof}[Proof of \cref{l4}]
Recall $|Z_x|\leq 3x$.
Let $r=\frac{1}{\sqrt{n}} \langle \sqrt{2h} D Z_x , X_{1}\rangle $. Then, by \cref{eq:2_2} and Taylor's expansion,
\begin{equation}
    \begin{aligned}
       \wt \mu_{1}= \frac{\IE^{Z_{x}} D X_{1} e^{r}}{\IE^{Z_x} e^{r} } =\frac{ \IE^{Z_x}
		   D X_{1}\left\{1+r+ \frac{1}{2}r^{2} + R_{1}\right\} }{\IE^{Z_x} \{1+r+ R_{2}\} }=\frac{ \IE^{Z_x}
            \left\{D X_{1}r+\frac{1}{2} D X_{1} r^{2} +D X_{1} R_{1}\right\} }{\IE^{Z_x} \{1 +
        R_{2}\} },
    \end{aligned}
    \label{eq:lem2_1}
\end{equation}
where $R_{1}= \frac{1}{2} \int_{0}^{1} (1-u)^{2} r^{3} e^{ur} du$, and $R_2= \int_{0}^{1} (1-u) r^{2} e^{ur} du$. 
We observe that
\be{
\IE ^{Z_x}( D X_1 r)=D \frac{\sqrt{2h} D Z_x}{\sqrt{n}},\ \IE ^{Z_x} \left(\frac{1}{2} D X_1
	r^2\right)=\frac{1}{2n}D \IE^{Z_{x}} \{\langle \sqrt{2h} D  Z_x ,  
X_{1}\rangle ^{2}  X_{1}\}.
}
Because of the assumption $\E e^{t_0|X_1|}\leq c_0<\infty$ and $|Z_x|\leq 3x$,
for $1<x\leq \varepsilon n^{1/6}$ with a sufficiently small $\varepsilon>0$, we have, $\E^{Z_x}(R_2)=O\bigl(\frac{x^2}{n}\bigr)$ and each component of $\IE ^{Z_x}(X_1 R_1)$ is $O\bigl(\frac{x^3}{n^{3/2}}\bigr)$.
Thus, 
\begin{equation}
    \begin{aligned}
       \wt \mu_{1}=\frac{\sqrt{2h}Q Z_x}{\sqrt{n}}+\frac{1}{2n}D\IE^{Z_{x}} \{\langle \sqrt{2h} DZ_x ,
X_{1}\rangle ^{2}X_{1}\} +DV, 
    \end{aligned}
    \label{eq:lem2_2}
\end{equation}
where each component of the $d$-vector $V$ is $O\left(\frac{x^3}{n^{3/2}}\right)$.
Next, for $\wt \Sigma$, by Taylor's expansion,
\begin{equation}
    \begin{aligned}
        \wt \Sigma= &\IE^{Z_x} \widetilde{X}_{1} \widetilde{X}_{1}^{T}- \wt\mu_{1} \wt\mu_{1}^{T} =
		\frac{D\left(\IE^{Z_x} X_{1} X_{1}^{T} e^{r}\right)D}{\IE^{Z_x} e^{r}}- \wt\mu_{1} \wt\mu_{1}^{T}\\
		=&\frac{D \left(\IE^{Z_x} X_{1} X_{1}^{T} (1+r+R_{2})\right)D}{\IE^{Z_x} (1+r+R_{2})}- \wt\mu_{1} \wt\mu_{1}^{T}.
    \end{aligned}
    \label{eq:lem2_3}
\end{equation}
Using similar arguments to control error terms as for
\cref{eq:lem2_2},
because 
\be{
\IE^{Z_x}\left(X_1 X_1^T\right)=I_d, \ \IE^{Z_x}\left(X_1 X_1^T r\right)=\frac{1}{\sqrt{n}} \IE^{Z_{x}}\left\{\langle \sqrt{2h} Z_x, X_{1}\rangle 
X_{1}X_{1}^{T} \right\},\ \IE ^{Z_x}r=0,
}
we have
\begin{equation}
    \begin{aligned}
		\wt \Sigma = D\biggl( I_{d} + \frac{1}{\sqrt{n}} \IE^{Z_{x}}\left\{\langle \sqrt{2h} DZ_x , X_{1}\rangle 
		X_{1}X_{1}^{T} \right\}+R\biggr)D,
    \end{aligned}
    \label{eq:lem2_4}
\end{equation}
where $R$ is a matrix such that each of its entries is $O\left(\frac{x^2}{n}\right)$ and $DRD$ absorbs $\wt\mu_{1} \wt\mu_{1}^{T}$.

From simple calculations similar to those in \eq{eq:lem2_1}--\eq{eq:lem2_4}, we obtain \eq{f26} and \eq{f30}.

Let $A=\IE^{Z_{x}}\left\{\langle \sqrt{2h} D Z_x , X_{1}\rangle 
    X_{1}X_{1}^{T} \right\}$, and let $A_{ij}$ and $R_{ij}$ be the $(i,j)th$ element of matrices
    $A$ and $R$, respectively. 
    Then, from the definition of determinate, $A_{ij}=O(x)$ and $R_{ij}=O\left(\frac{x^2}{n}\right)$, we have
    \begin{equation}
        \begin{aligned}
			\det \left( D^{-1}\wt\Sigma D^{-1} \right)= \sum_{\sigma}
            (-1)^{\sgn(\sigma)}\prod_{i=1}^{d}\left(\delta_{i\sigma(i)}
            +\frac{1}{\sqrt{n}} A_{i\sigma(i)}+
            R_{i\sigma(i)}\right)=\det \left(I_{d}+ \frac{A}{\sqrt{n}}\right)+ O\Bigl(\frac{x^{2}}{n}\Bigr),
        \end{aligned}
        \label{eq:lem2_5}
    \end{equation}
    where the sum is over all permutations $\sigma$ of $\{1,\dots, n\}$, $\sgn$ denotes the sign of a permutation, $\delta_{ij}=1$ for $i=j$ and $\delta_{ij}=0$ for $i\ne j$. Moreover, because $\wt \lambda_j=O(x)$, we have
    \begin{equation}
        \begin{aligned}
            \det \left(I_{d}+ \frac{A}{\sqrt{n}}\right)= \prod_{j=1}^{d} \left(1+\frac{1}{\sqrt{n}} \wt \lambda_{j}\right)
            =1+\frac{1}{\sqrt{n}}\sum_{j=1}^{d}  \wt \lambda_{j} + O\left(\frac{x^{2}}{n}\right).
        \end{aligned}
        \label{eq:lem2_6}
    \end{equation}
    Combining \cref{eq:lem2_5,eq:lem2_6}, we obtain \eq{eq:lem2_9} for $\det \wt \Sigma$.
    
    For small enough $x^{6}/n$, the absolute values of eigenvalues of matrices $A/\sqrt{n}$ and $R$ are
    smaller than $1/4$, and thus we have 
    \begin{equation}
        \begin{aligned}
			\wt \Sigma^{-1}= D^{-1} \Bigl(I_{d} - \frac{1}{\sqrt{n}} A+ R'\Bigr)D^{-1},
        \end{aligned}
        \label{eq:lem2_10}
    \end{equation}
   where 
   \begin{equation}
       \begin{aligned}
           R' = \sum^{\infty}_{r=2} (-1)^{r}
           \left(\frac{1}{\sqrt{n}} A+R \right)^{r}-R.
       \end{aligned}
       \label{eq:lem2_11}
   \end{equation}
   For any two vectors $X$ and $Y$ with $\lvert X\rvert=\lvert Y\rvert=1$, we have
   \begin{equation}
       \begin{aligned}
           \lvert X^{T}R'Y \rvert\leq& \Bigl\lvert\Bigl\{
           \Bigl(\frac{1}{\sqrt{n}}A+R\Bigr)^{2}\Bigr\}^{T}X \Bigr\rvert \,\sum^{\infty}_{r=0} \Bigl\lvert
           \Bigl(\frac{1}{\sqrt{n}}A+R\Bigr)^{r} Y\Bigr\rvert + \lvert X^{T} R Y\rvert\\ \leq& \Bigl\lvert\Bigl\{
           \Bigl(\frac{1}{\sqrt{n}}A+R\Bigr)^{2}\Bigr\}^{T}X \Bigr\rvert \,\sum^{\infty}_{r=0} 
       \Bigl(\frac{1}{2}\Bigr)^{r}  + \lvert X^{T} R Y\rvert\\
               \leq& C \frac{x^{2}}{n},
       \end{aligned}
       \label{eq:lem2_12}
   \end{equation}
   which proves \cref{f24}. 
    Finally, \cref{eq:lem2_8}
    follows from the definition of eigenvalues.
  \end{proof}

 \begin{proof}[Proof of \cref{lem:4}]
  Recalling that $\wt p(y)= \phi( \wt \Sigma^{-1/2}y)/ \sqrt{\det \wt \Sigma}$, by \cref{eq:lem2_9} and $\wt \lambda_j=O(x)$, we
  have, for sufficiently small $\varepsilon>0$ and $1<x\leq \varepsilon n^{1/6}$,
  \begin{equation}
      \begin{aligned}
		  \wt p(y)&= \det (Q)^{-1/2} \Biggl[1-\frac{1}{2\sqrt{n}} \sum^{d}_{j=1} \wt \lambda_{j}+
          O\left(\frac{x^{2}}{n}\right) \Biggr] \phi( \wt \Sigma^{-1/2}y)\\&= 
              \Biggl[1-\frac{1}{2\sqrt{n}} \sum^{d}_{j=1}\wt \lambda_{j}+
				  O\left(\frac{x^{2}}{n}\right) \Biggr]\frac{1}{\sqrt{\det Q}(2\pi)^{d/2}} \exp \Bigl\{
          -\frac{1}{2} y^{T} \wt \Sigma^{-1} y\Bigr\}.
     \end{aligned}
      \label{eq:lem4_2}
  \end{equation}
  By \eq{f24} and \eq{f25}, we have 
  \begin{equation}
      \begin{aligned}
		  y^{T} \wt \Sigma^{-1} y = y^{T}Q^{-1}y-\frac{1}{\sqrt{n}} \IE^{Z_x} \langle
		  D^{-1} y, X_{1} \rangle^{2} \langle \sqrt{2h} D Z_x
		  , X_{1}\rangle + O \Bigl(\frac{x^{2}\lvert D^{-1} y\rvert^{2}}{n}\Bigr),
      \end{aligned}
      \label{eq:lem4_3}
  \end{equation}
  \begin{equation}
      \begin{aligned}
          y^{T} \wt \Sigma^{-1} (\sqrt{n}\wt\mu_{1})=\langle \sqrt{2h} Z_x ,
		   y\rangle -\frac{1}{2\sqrt{n}} \IE^{Z_x} \langle D^{-1} y, X_{1} \rangle  \langle
		   \sqrt{2h} D Z_x
		  , X_{1}\rangle^{2} + O \Bigl(\frac{x^{3}\lvert D^{-1} y\rvert}{n}\Bigr),\\
           \text{(Two terms of order $\frac{1}{\sqrt{n}}$ are combined)}
      \end{aligned}
      \label{eq:lem4_4}
  \end{equation}
  \begin{equation}
      \begin{aligned}
          (\sqrt{n}\wt\mu_{1})^{T} \wt \Sigma^{-1} (\sqrt{n}\wt\mu_{1})= 2h \lvert
		  D Z_{x}\rvert^{2} + O \Bigl(\frac{x^{4}}{n}\Bigr). \quad \text{(Two terms of order $\frac{1}{\sqrt{n}}$ are cancelled)}
      \end{aligned}
      \label{eq:lem4_5}
  \end{equation}
From \eq{f25},
\ben{\label{f31}
\sqrt{n}\wt \mu_1=\sqrt{2h} Q Z_x+DV',\quad \text{where each component of the $d$-vector $V'$ is $O\left(\frac{x^2}{\sqrt{n}}\right)$}.
}
From \eq{f24},
\ben{\label{f32}
	\wt \Sigma^{-1} =Q^{-1}+D^{-1}R''D^{-1},\quad \text{where each entry of $d\times d$ matrix $R''$ is $O\left(\frac{x}{\sqrt{n}}\right)$}.
}
From \eq{f31}, \eq{f32} and \eq{f26},  
  with $\widehat{X}_{1}=\widetilde{X}_{1}-\wt \mu_{1}$, we have, by only keeping the main term (recall there will be a factor of $1/\sqrt{n}$ in front of the second term on the left-hand side of \eq{eq:lem4_1})
  \begin{equation}
      \begin{aligned}
          &\quad\IE^{Z_x} \left\{ 3\langle \wt \Sigma^{-1} \wh X_{1}, \wh X_{1}\rangle\langle \wt
              \Sigma^{-1}(y-\sqrt{n} \wt \mu_{1}), \wh {X}_{1}\rangle-
          \langle \wt \Sigma^{-1}(y-\sqrt{n} \wt \mu_{1}), \wh X_{1}\rangle^3 \right\}
		\\&= \E^{Z_x} \biggl\{ 3\Bigl\langle D^{-1}\Bigl(I_d+O\Bigl(\frac{x}{\sqrt{n}}\Bigr)\Bigr)
			D^{-1}\Bigl(\wt X_1+D\cdot O\Bigl(\frac{x}{\sqrt{n}}\Bigr)\Bigr), (\wt X_1+D\cdot O\Bigl(\frac{x}{\sqrt{n}}\Bigr))\Bigr\rangle 
           \\&\qquad \quad \quad \times\Bigl\langle  D^{-1}\Bigl(I_d+O\Bigl(\frac{x}{\sqrt{n}}\Bigr)\Bigr) D^{-1}\Bigl(y-\sqrt{2h}QZ_x+D\cdot O\Bigl(\frac{x^2}{\sqrt{n}}\Bigr)\Bigr) ,\Bigl(\wt X_1+D\cdot O\Bigl(\frac{x}{\sqrt{n}}\Bigr)\Bigr) \Bigr\rangle 
           \\&\quad\quad\quad\quad  -\Bigl\langle  D^{-1}\Bigl(I_d+O\Bigl(\frac{x}{\sqrt{n}}\Bigr)\Bigr) D^{-1}\Bigl(y-\sqrt{2h}QZ_x+D\cdot O\Bigl(\frac{x^2}{\sqrt{n}}\Bigr)\Bigr) , \Bigl(\wt X_1+D\cdot O\Bigl(\frac{x}{\sqrt{n}}\Bigr)\Bigr)\Bigr\rangle^3 \biggr\}
          \\&=  \IE^{Z_x}  \left\{ 3\langle   X_{1}, X_{1}\rangle\langle
			  D^{-1}y-\sqrt{2h}D Z_x,  {X}_{1}\rangle-
		  \langle D^{-1} y-\sqrt{2h} DZ_x,  X_{1}\rangle^3 \right\}+
		  O\biggl(\frac{x^{4}}{\sqrt{n}}+ \frac{x \lvert D^{-1} y\rvert^{3}}{\sqrt{n}}\biggr),
      \end{aligned}
      \label{eq:lem4_6}
  \end{equation}
  where we used \eq{f31}, \eq{f32} and an abuse of notation (using $O(\cdot)$ for vectors and matrices to show the magnitude of their entries) in the first equality, and \eq{f26} and straightforward simplifications of error terms in the second equality.
  For example, one of the error terms is
  \be{
  \E^{Z_x} \left\{ \langle D^{-1} I_d D^{-1} (y-\sqrt{2h} Q Z_x), \wt X_1 \rangle^2\langle D^{-1} I_d D^{-1} (y-\sqrt{2h} Q Z_x) , D\cdot O(\frac{x}{\sqrt{n}})\rangle    \right\},
  }
  which is of the order
  \be{
  O(x^3+|D^{-1}y|^3) \frac{x}{\sqrt{n}} \E^{Z_x} |D^{-1} \wt X_1|^2=O(\frac{x^4+x|D^{-1}y|^3}{\sqrt{n}})   \qquad (\text{cf. } \eq{eq:lem2_3}\&\eq{eq:lem2_4}).
  }
  By \cref{eq:lem4_7,eq:2_3,eq:lem4_2}, we have 
  \begin{equation}
      \begin{aligned}
      &\quad \wt p(y-\sqrt{n}
                  \wt  \mu_{1} ) +\frac{1}{6\sqrt{n}}\IE^{Z_{x}}\Bigl\{\wt p'''(y-\sqrt{n}
                    \wt    \mu_{1})
                      (\widetilde{X}_{1}-\wt \mu_{1})^3\Bigr\}\\
          &=\Bigl[1-\frac{1}{2\sqrt{n}} \sum^{d}_{j=1}\wt \lambda_{j}+
			  O\Bigl(\frac{x^{2}}{n}\Bigr) \Bigr]\frac{1}{\sqrt{\det Q}(2\pi)^{d/2}} \exp \Bigl\{
              -\frac{1}{2} (y-\sqrt{n}\wt \mu_{1} )^{T} \wt \Sigma^{-1} (y-\sqrt{n}
         \wt \mu_{1})\Bigr\}\\
          &\quad\times \Bigl(1+\frac{1}{6\sqrt{n}}\IE^{Z_x} \left\{ 3\langle \wt \Sigma^{-1} \wh X_{1}, \wh X_{1}\rangle\langle \wt
              \Sigma^{-1}(y-\sqrt{n} \wt \mu_{1}), \wh {X}_{1}\rangle-
          \langle \wt \Sigma^{-1}(y-\sqrt{n} \wt \mu_{1}), \wh X_{1}\rangle^3 \right\}\Bigr).     
      \end{aligned}
      \label{eq:lem4_8}
  \end{equation}
   Combining
              \cref{eq:lem4_3,eq:lem4_4,eq:lem4_5,eq:lem4_6,eq:lem4_8}, we have
              \begin{equation}
                  \begin{aligned}
                      &\quad \wt p(y-\sqrt{n}
                 \wt   \mu_{1} ) +\frac{1}{6\sqrt{n}}\IE^{Z_{x}}\Bigl\{\wt p'''(y-\sqrt{n}
                   \wt     \mu_{1})
                      (\widetilde{X}_{1}-\wt \mu_{1})^3\Bigr\}\\&= \exp \bigl\{-h \lvert D Z_{x}\rvert^{2}+\langle  \sqrt{2h} Z_x , y\rangle 
					  \bigr\}\phi(D^{-1} y)(\det D)^{-1} \\&\quad\times\Bigl(1+B_{0}+
                      O\Bigl(\frac{x^{2}}{n}\Bigr)\Bigr)\exp \Bigl\{B_{1}+O\Bigl(\frac{x^{4}}{n}+
					  \frac{x^{2}\lvert D^{-1} y\rvert^{2}}{n}\Bigr)\Bigr\}\biggl(1+B_{2}+
				  O\biggl(\frac{x^{4}}{n}+ \frac{x \lvert D^{-1} y\rvert^{3}}{n}\biggr)\biggr)\\&= H_{1}(y)+H_{2}(y),
                  \end{aligned}
                  \label{eq:lem4_10}
              \end{equation}
              where $H_{1}(y)$ is defined in \cref{eq:lem4_11}
              and 
              \begin{equation}
                  \begin{aligned}
                      H_{2}(y)&=  \exp \bigl\{-h \lvert D Z_{x}\rvert^{2}+ \langle \sqrt{2h} Z_x, y\rangle 
						  \bigr\}\phi(D^{-1}y)(\det D)^{-1} \\&\quad\times\Bigl(1+B_{0}+
                      O\Bigl(\frac{x^{2}}{n}\Bigr)\Bigr)\biggl(\exp \biggl\{B_{1}+O\Bigl(\frac{x^{4}}{n}+
						  \frac{x^{2}\lvert D^{-1} y\rvert^{2}}{n}\Bigr)\biggr\}-B_{1}-1+O\Bigl(\frac{x^{4}}{n}+
					  \frac{x^{2}\lvert D^{-1} y\rvert^{2}}{n}\Bigr)\biggr)\\&\quad\times\biggl(1+B_{2}+
				  O\biggl(\frac{x^{4}}{n}+ \frac{x \lvert D^{-1} y\rvert^{3}}{n}\biggr)\biggr).
                  \end{aligned}
                  \label{eq:lem4_12}
              \end{equation}   
			  Next, consider $H_{2}(y)$. By $B_1=O\left(\frac{x|D^{-1}y|^2+x^2|D^{-1} y|}{\sqrt{n}}\right)$, the elementary inequality $|e^{x}-1-x|\leq \frac{1}{2}
			  x^{2} e^{\lvert x\rvert}$, and recalling that $1< x\leq \varepsilon{n^{1/6}}$, we have
              \begin{equation}
                  \begin{aligned}
                    &  \quad\biggl\lvert\exp \biggl\{B_{1}+O\biggl(\frac{x^{4}}{n}+
      \frac{x^{2}\lvert  D^{-1} y\rvert^{2}}{n}\biggr)\biggr\}-B_{1}-O\biggl(\frac{x^{4}}{n}+
                      \frac{x^{2}\lvert  D^{-1} y\rvert^{2}}{n}\biggr)-1\biggr\rvert\\&\leq \frac{1}{2} \biggl(B_{1}+O\biggl(\frac{x^{4}}{n}+
                      \frac{x^{2}\lvert  D^{-1} y\rvert^{2}}{n}\biggr)\biggr)^{2} \exp \biggl\{ \biggl\lvert
                          B_{1}+O\biggl(\frac{x^{4}}{n}+
  \frac{x^{2}\lvert  D^{-1} y\rvert^{2}}{n}\biggr)\biggr\rvert\biggr\}\\&\leq  C
  \biggl(\frac{x^{2}|D^{-1}y|^{4}}{n}+\frac{x^{4}|D^{-1}y|^{2}}{n} + \frac{x^{2}}{n}\biggr) \exp \biggl\{
                      C\biggl(\frac{x \lvert  D^{-1} y\rvert^{2}+x^{2} \lvert  D^{-1} y\rvert}{\sqrt{n}}\biggr)\biggr\}
                 \end{aligned}
                  \label{eq:lem4_13}
              \end{equation}
              for some positive constant $C$. By \cref{eq:lem4_12,eq:lem4_13}, we complete the proof of
              \cref{eq:lem4_14}.

          \end{proof}  
  
\begin{proof}[Proof of \cref{l1}]
Because of the assumption of the finiteness of the moment generating function of $X_1$ near 0, the
function $\E e^{\langle a, X_{1}\rangle }$ is finite for all $a\in \mathbb{R}^d$ with $|a|\leq
\varepsilon $ for a sufficiently small $\varepsilon>0$. 
 Recall $\hat{G}(b)=\IE  e^{\langle b, X_{1}\rangle } $ and we
have 
\begin{equation}
    \begin{aligned}
        \IE e^{\langle a, W\rangle } = \hat{G}^{n} \Bigl(\frac{a}{\sqrt{n}}\Bigr).
    \end{aligned}
    \label{eq:lem3_1}
\end{equation}
Furthermore, 
\begin{equation}
    \begin{aligned}
        &\quad\hat{G}^{n} \Bigl(\frac{a}{\sqrt{n}}\Bigr)- \exp
        \biggl\{\frac{|a|^2}{2}+\frac{1}{6\sqrt{n}} \IE \langle a,
        X_{1}\rangle ^{3}\biggr\}\\
        &= \exp
        \biggl\{\frac{|a|^2}{2}+\frac{1}{6\sqrt{n}} \IE \langle a,
        X_{1}\rangle ^{3}\biggr\}\\
        &\quad\times \biggl(\exp \biggl\{n \biggl(\log \hat{G}\Bigl(\frac{a}{\sqrt{n}}\Bigr)- \frac{|a|^2}{2n}-
        \frac{\IE \langle a, X_{1}\rangle ^{3}}{6n^{3/2}}\biggr) \biggr\}-1\biggr).
    \end{aligned}
    \label{eq:lem3_2}
\end{equation}
 By Taylor's expansion and using $\IE X_1=0$ and $\Cov(X_1)=I_d$,
\begin{equation}
    \begin{aligned}
        \log \hat{G}\Bigl(\frac{a}{\sqrt{n}}\Bigr)- \frac{|a|^2}{2n}-
        \frac{\IE \langle a, X_{1}\rangle ^{3}}{6n^{3/2}}=\frac{1}{6} \int_{0}^{1} (1-u)^{3}
        \biggl(\frac{d^{4}}{du^{4}} \log \hat{G}\Bigl(\frac{ua}{\sqrt{n}}\Bigr) \biggr)     
         du.
    \end{aligned}
    \label{eq:lem3_3}
\end{equation}
To bound the integration on the right-hand side of \cref{eq:lem3_3}, we need the following lemma:
\begin{lemma}
    For $a\in \mathbb{R}^d$ such that $|a|\leq t_{0}\sqrt{n}/2$ and  $\Bigl\lvert \hat{G}\Bigl(\frac{ua}{\sqrt{n}}\Bigr)\Bigr\rvert\geq
    \frac{1}{2}, \forall\ u\in [0,1]$, we have for $u\in [0,1]$, 
    \begin{equation}
        \begin{aligned}
            \biggl\lvert\frac{d^{4}}{du^{4}} \log \hat{G}
            \biggl(\frac{ua}{\sqrt{n}}\biggr)\biggr\rvert\leq \frac{C}{n^{2}} ( \lvert
                a\rvert^{4}),
        \end{aligned}
        \label{eq:lem5_1}
    \end{equation}
    where $C$ is a constant depending only on $d$, $c_{0}$, and $t_{0}$ in \cref{t1}.
    \label{lem:5.1}
\end{lemma}
\begin{proof}
\begin{equation}
    \begin{aligned}
        \frac{d^{4}}{du^{4}} \log \hat{G}
\biggl(\frac{ua}{\sqrt{n}}\biggr)= \sum^{ }_{}  \mathrm{c}\left(\left\{\beta_{1}, \ldots,
\beta_{j}\right\}\right) \frac{\frac{d^{\beta_{1}}}{du^{\beta_{1}}}
\hat{G}\Bigl(\frac{ua}{\sqrt{n}}\Bigr)\cdots\frac{d^{\beta_{j}}}{du^{\beta_{j}}}
\hat{G}\Bigl(\frac{ua}{\sqrt{n}}\Bigr)}{\hat{G}\Bigl(\frac{ua}{\sqrt{n}}\Bigr)^{j}},
\end{aligned}
    \label{eq:lem5_2}
\end{equation}
where the summation is over all collections of nonnegative integers  $\left\{\beta_1, \ldots, \beta_j\right\}$ satisfying
$$
\beta_{1}+\cdots+\beta_{j}=4, \quad 1 \leqslant j \leqslant 4,
$$
and the constant $\mathrm{c}\left(\left\{\beta_{1}, \ldots, \beta_{j}\right\}\right)$ depends only
on the collection $\left\{\beta_{1}, \ldots, \beta_{j}\right\}$. For any nonnegative integer
$\beta\leq 4$, we have 
\begin{equation}
    \begin{aligned}
        \Bigl\lvert \frac{d^{\beta}}{du^{\beta}}
        \hat{G}\Bigl(\frac{ua}{\sqrt{n}}\Bigr)\Bigr\rvert = \biggl\lvert\frac{1}{\sqrt{n}^{\beta}}\IE \left\{\langle a ,
    X_{1}\rangle^\beta e^{\frac{u}{\sqrt{n}} \langle a, X_1\rangle }\right\}\biggr\rvert\leq
    \frac{1}{\sqrt{n}^{\beta}}\IE \left\{(|a|
    |X_{1}|)^\beta e^{\frac{u|a|}{\sqrt{n}}  |X_1|}\right\}.
\end{aligned}
    \label{eq:lem5_3}
\end{equation}
Therefore, for $|a|\leq t_{0}\sqrt{n}/2$ and $u\in [0,1]$, we have that \cref{eq:lem5_3} can be bounded by $\frac{|a|^\beta}{(\sqrt{n})^\beta}$ multiplied by a constant depending only on $d$, $c_{0}$, and $t_{0}$ in \cref{t1}. Combining
\cref{eq:lem5_2}, \cref{eq:lem5_3}, and the condition $\Bigl\lvert \hat{G}\Bigl(\frac{ua}{\sqrt{n}}\Bigr)\Bigr\rvert\geq
\frac{1}{2}, \forall\ u\in [0,1]$ we complete the proof.
\end{proof}
By Taylor's expansion, we have, $\forall\ u\in [0,1]$,
\begin{equation}
    \begin{aligned}
        \Bigl\lvert \hat{G}\Bigl(\frac{ua}{\sqrt{n}}\Bigr)-1\Bigr\rvert =  \Bigl\lvert \IE
        \int_{0}^{1} \frac{1}{\sqrt{n}}\langle ua, X_{1}\rangle  \exp \Bigl\{\frac{v}{\sqrt{n}}\{\langle ua,
    X_{1}\rangle \}\Bigr\} dv\Bigr\rvert\leq\frac{\lvert a\rvert}{\sqrt{n}}   \IE \lvert
    X_{1} \rvert e^{\frac{|a|}{\sqrt{n}}  |X_1|}.
    \end{aligned}
    \label{eq:lem3_4}
\end{equation}
Therefore, there exists a constant $\varepsilon>0$ such that for $\lvert a\rvert \leq
\varepsilon \sqrt{n}$, \cref{eq:lem3_4} is less than $1/2$ and 
\begin{equation}
    \begin{aligned}
       \Bigl\lvert \hat{G}\Bigl(\frac{ua}{\sqrt{n}}\Bigr)\Bigr\rvert\geq
\frac{1}{2}, \forall\ u\in [0,1].
    \end{aligned}
    \label{eq:lem3_5}
\end{equation}
By \cref{eq:lem3_3}, \cref{eq:lem3_5}, and \cref{lem:5.1}, we have 
\begin{equation}
    \begin{aligned}
        \Bigl\lvert  \log \hat{G}\Bigl(\frac{a}{\sqrt{n}}\Bigr)- \frac{|a|^2}{2n}-
        \frac{\IE \langle a, X_{1}\rangle^{3}}{6n^{3/2}}\Bigr\rvert\leq \frac{C}{n^{2}} \lvert
        a\rvert^{4}.
    \end{aligned}
    \label{eq:lem3_6}
\end{equation}
For the second factor on the right hand side of \cref{eq:lem3_2}, from the elementary inequality 
\begin{equation}
    \begin{aligned}
        |\exp(x)-1|\leq |x| \exp(\lvert x\rvert)
    \end{aligned}
    \label{eq:lem3_7}
\end{equation}
and \cref{eq:lem3_6}, we have 
\begin{equation}
    \begin{aligned}
       &\quad \biggl\lvert\exp \Bigl\{n \Bigl(\log \hat{G}\Bigl(\frac{a}{\sqrt{n}}\Bigr)- \frac{|a|^2}{2n}-
        \frac{\IE \langle a, X_{1}\rangle^{3}}{6n^{3/2}}\Bigr) \Bigr\}-1\biggr\rvert
                    \leq  \frac{C}{n}  \lvert a\rvert^{4}  
            \exp \Bigl\{\frac{C}{n}  \lvert a\rvert^{4}  \Bigr\}.
    \end{aligned}
    \label{eq:lem3_8}
\end{equation}
From \cref{eq:lem3_2,eq:lem3_8}, we have 
\begin{equation}
    \begin{aligned}
        &\quad\biggl\lvert\hat{G}^{n} \Bigl(\frac{a}{\sqrt{n}}\Bigr)- \exp
        \Bigl\{\frac{|a|^2}{2}+\frac{\IE \langle a,
        X_{1}\rangle ^{3}}{6\sqrt{n}} \Bigr\}\biggr\rvert \leq \frac{C}{n}  \lvert a\rvert^{4}  
             \exp \Bigl\{\frac{C}{n}  \lvert a\rvert^{4}  + \frac{|a|^2}{2}+\frac{C|a|^3}{6\sqrt{n}}\Bigr\}.
    \end{aligned}
    \label{eq:lem3_9}
\end{equation}
Next, we give the following bound. By Taylor's expansion
\begin{equation}
    \begin{aligned}
     &\quad  \biggl\lvert \exp \Bigl\{\frac{|a|^2}{2}+\frac{\IE\langle a,
        X_{1}\rangle^{3}}{6\sqrt{n}} \Bigr\}-\exp
        \Bigl\{\frac{|a|^2}{2} \Bigr\}\biggl(1+\frac{\IE\langle a,
        X_{1}\rangle^{3}}{6\sqrt{n}}\biggr)\biggr\rvert\\
        &=  \exp
        \Bigl\{\frac{|a|^2}{2} \Bigr\} \biggl\lvert\exp \Bigl\{\frac{\IE\langle a,
    X_{1}\rangle^{3}}{6\sqrt{n}} \Bigr\}-1-\frac{\IE\langle a,
    X_{1}\rangle^{3}}{6\sqrt{n}}\biggr\rvert \\
      &= \exp
        \Bigl\{\frac{|a|^2}{2} \Bigr\} \biggl\lvert
        \int_{0}^{1} (1-u)\exp \Bigl\{\frac{u\IE \langle a,
                X_{1}\rangle^{3}}{6\sqrt{n}} \Bigr\} \biggl(\frac{\IE \langle a,
    X_{1}\rangle^{3}}{6\sqrt{n}}\biggr)^{2}du \biggr\rvert
   \\ &\leq  \frac{C}{n}  \lvert a\rvert^{6}  \exp
       \Bigl\{\frac{\lvert a\rvert^2}{2} 
           +\frac{C}{\sqrt{n}}  \lvert a\rvert^{3} \Bigr\}.  
    \end{aligned}
    \label{eq:lem3_13}
\end{equation}
Combining \eq{eq:lem3_9} and \eq{eq:lem3_13}, we have for $\lvert a\rvert \leq
\varepsilon\sqrt{n} $,
\begin{equation}
    \begin{aligned}
      &\quad  \biggl\lvert\hat{G}^{n} \Bigl(\frac{a}{\sqrt{n}}\Bigr)- \exp
        \Bigl\{\frac{|a|^2}{2} \Bigr\}\biggl(1+\frac{\IE\langle a,
    X_{1}\rangle^{3}}{6\sqrt{n}}\biggr)\biggr\rvert\\
    &\leq \biggl(\frac{C}{n}  \lvert a\rvert^{4} +\frac{C}{n}  \lvert a\rvert^{6}\biggr) \exp
            \Bigl\{  \frac{\lvert a\rvert^{2}}{2}
                +
            \frac{C}{\sqrt{n}}  \lvert a\rvert^{3} + \frac{C}{n} \lvert
        a\rvert^{4} \Bigr\},
    \end{aligned}
    \label{eq:lem3_15}
\end{equation}
which is the required result.
\end{proof}

\begin{proof}[Proof of \cref{lem:8}]
The case $d=1$ follows from the integration by parts formula and the asymptotic tail probability of the $\chi_1$ distribution.
In the following, we consider the case $d\geq 2$.

First, we have 
\begin{equation}
	\begin{aligned}
	&	\int_{\lvert D y\rvert> x}^{ } \lvert y\rvert^{r} 
                          \exp \biggl\{-\frac{\lvert y\rvert^{2}}{2}+
                          C\frac{x \lvert y\rvert^{2}}{\sqrt{n}}\biggr\}
		dy\\=&
		\int_{y\in \mathbb{R}^{d}} 1_{	\{|Dy|> x\}} \lvert y\rvert^{r} 
                          \exp \biggl\{-\frac{\lvert y\rvert^{2}}{2}+
                          C\frac{x \lvert y\rvert^{2}}{\sqrt{n}}\biggr\}
                                   dy
		\\=& \int_{u\geq 0} \int_{y\in \partial B(0,u)} 1_{	\{|Dy|> x\}} u^{r} \exp \biggl\{-\frac{u^{2}}{2}+
                          C\frac{x u^{2}}{\sqrt{n}}\biggr\}dSdu
			\\=& \int_{u\geq 0}  	S(\{|Dy|> x\}\cap \{y\in \partial B(0,u)\}) u^{r} \exp \biggl\{-\frac{u^{2}}{2}+
                          C\frac{x u^{2}}{\sqrt{n}}\biggr\}du,
	\end{aligned}
	\label{eq:lem8_1}
\end{equation}
where $S(\cdot)$ is the Lebesgue measure of $(d-1)$-dimensional surface. Let 
\begin{equation}
	\begin{aligned}
		\xi(a)= \frac{S(\{|Dy|> x\}\cap \{y\in \partial B(0,ax)\})}{S(\{y\in \partial B(0,ax)\})}, \quad a>0.
	\end{aligned}
	\label{eq:lem8_2}
\end{equation}
We can easily verify that $\xi(a)$ does not depend on $x$ and $\xi(a)$ is a continuous and increasing function such
that
\begin{equation}
	\begin{aligned}
		\xi(a)=0 \mbox{ for $0<a\leq 1$, and } \xi(a)=1 \mbox{ for $a\geq
		\frac{1}{q_{d}^{1/2}}$},
	\end{aligned}
	\label{eq:lem8_3}
\end{equation}
where $q_{d}$ is the smallest eigenvalue of $Q$. 
Let $y_{1}$ be the first component of a vector $y\in
\mathbb{R}^{d}$. There exists an absolute constant $\delta> 0$ (in particular, it does not depend on $Q$) such that 
\begin{equation}
	\begin{aligned}
		\frac{S(\{y_{1}> x\}\cap \{y\in \partial B(0,(1+\delta)x)\})}{S(\{y\in \partial
		B(0,(1+\delta)x)\})}=\frac{1}{16}.
	\end{aligned}
	\label{eq:lem8_4}
\end{equation}
Because $\{|Dy|> x\}\supset \{|y_1|>x\}$ (recall the largest eigenvalue of $Q$ is $1$), we then have 
\begin{equation}
	\begin{aligned}
		\xi(1+\delta)&= \frac{S(\{|Dy|> x\}\cap \{y\in \partial B(0,(1+\delta)x)\})}{S(\{y\in \partial
				B(0,(1+\delta)x)\})}\\&\geq 2\frac{S(\{y_{1}> x\}\cap \{y\in \partial B(0,(1+\delta)x)\})}{S(\{y\in \partial
		B(0,(1+\delta)x)\})}\\&=\frac{1}{8}.
	\end{aligned}
	\label{eq:lem8_5}
\end{equation}
We now return to \cref{eq:lem8_1}. By \cref{eq:lem8_1,eq:lem8_2}, we observe that 
\begin{equation}
	\begin{aligned}
		\int_{\lvert D y\rvert> x}^{ } \lvert y\rvert^{r} 
                          \exp \biggl\{-\frac{\lvert y\rvert^{2}}{2}+
                          C\frac{x \lvert y\rvert^{2}}{\sqrt{n}}\biggr\}
						  dy= \frac{2 \pi^{d/2}}{\Gamma(\frac{d}{2})} \int_{u> x}
						  \xi\biggl(\frac{u}{x}\biggr)	u^{r+d-1} \exp \biggl\{-\frac{u^{2}}{2}+
                          C\frac{x u^{2}}{\sqrt{n}}\biggr\}du,
	\end{aligned}
	\label{eq:lem8_6}
\end{equation}
where we use the fact that the surface area of the $d$-dimensional unit ball is $2
\pi^{d/2}/\Gamma(\frac{d}{2})$.
Next, we deal with the integration on the right-hand side of \cref{eq:lem8_6}. By a change of variable
and the integration by parts formula, we have, choosing $\varepsilon>0$ to be sufficiently small such that $1-\frac{2Cx}{\sqrt{n}}>\frac{1}{2}$,
\begin{equation}
	\begin{aligned}
		&\quad\int_{u>  (1+\delta)x}
						  \xi\biggl(\frac{u}{x}\biggr)	u^{r+d-1} \exp \biggl\{-\frac{u^{2}}{2}+
	C\frac{x u^{2}}{\sqrt{n}}\biggr\}du\\&=
          \int_{a> 1+\delta}
		  \xi(a)	(ax)^{r+d-1} \exp \biggl\{-\frac{(ax)^{2}}{2}+
		  C\frac{x (ax)^{2}}{\sqrt{n}}\biggr\}xda
									   \\&=- \Bigl(1-\frac{2Cx}{\sqrt{n}}\Bigr)^{-1} \xi(a)	(ax)^{r+d-2} \exp \biggl\{-\frac{(ax)^{2}}{2}+
		C\frac{x (ax)^{2}}{\sqrt{n}}\biggr\}\bigg|_{1+\delta}^{\infty}
       \\&\quad+\Bigl(1-\frac{2Cx}{\sqrt{n}}\Bigr)^{-1}\int_{a> 1+\delta}
		  	(ax)^{r+d-2} \exp \biggl\{-\frac{(ax)^{2}}{2}+
		  C\frac{x (ax)^{2}}{\sqrt{n}}\biggr\}d\xi(a)
		  \\&\quad+(r+d-2) \Bigl(1-\frac{2Cx}{\sqrt{n}}\Bigr)^{-1}\int_{a> 1+\delta}
		  \xi(a)	(ax)^{r+d-3} \exp \biggl\{-\frac{(ax)^{2}}{2}+
		  C\frac{x (ax)^{2}}{\sqrt{n}}\biggr\}xda
	   \\&:=J_{1}+J_{2}+J_{3}  .
	\end{aligned}
	\label{eq:lem8_7}
\end{equation}
Recalling that $\xi(a)$ is increasing, $1/8\leq \xi(a)\leq 1$ for $a\geq 1+\delta$, and $1-\frac{2Cx}{\sqrt{n}}>\frac{1}{2}$, we have 
\begin{equation}
	\begin{aligned}
		J_{1}+J_{2}&\leq C \xi(1+\delta)	( (1+\delta) x)^{r+d-2} \exp \biggl\{-\frac{(
			(1+\delta)x)^{2}}{2}+
		C\frac{x ( (1+\delta) x)^{2}}{\sqrt{n}}\biggr\}
\\&\quad+C(r)
( (1+\delta)x)^{r+d-2} \exp \biggl\{-\frac{( (1+\delta)x)^{2}}{2}+
C\frac{x ( (1+\delta) x)^{2}}{\sqrt{n}}\biggr\}\int_{a> 1+\delta}d\xi(a)
\\&\leq C(r) x^{r}
x^{d-2} \exp \biggl\{-\frac{( (1+\delta)x)^{2}}{2}
\biggr\}.
	\end{aligned}
	\label{eq:lem8_8}
\end{equation}
Repeating \cref{eq:lem8_7} with $C=0$ and $r=0$, we have 
\begin{equation}
	\begin{aligned}
		&\quad\int_{u>  (1+\delta)x}
						  \xi\biggl(\frac{u}{x}\biggr)	u^{d-1} \exp \biggl\{-\frac{u^{2}}{2}\biggr\}du
	  \\&=  \xi( 1+\delta)	( (1+\delta)x)^{d-2} \exp \biggl\{-\frac{( (1+\delta)x)^{2}}{2}\biggr\}
		\\&\quad+(d-2)\int_{a> 1+\delta}
		  \xi(a)	(ax)^{d-3} \exp \biggl\{-\frac{(ax)^{2}}{2}\biggr\}xda
       \\&\quad+\int_{a> 1+\delta}
		  	(ax)^{d-2} \exp \biggl\{-\frac{(ax)^{2}}{2}\biggr\}d\xi(a)\\
		 &\geq \frac{1}{8} x^{d-2} \exp \biggl\{-\frac{((1+\delta) x)^{2}}{2}\biggr\},
	\end{aligned}
	\label{eq:lem8_9}
\end{equation}
where we used the fact that the last two integrations in \cref{eq:lem8_9} are positive and $\xi(1+\delta)\geq \frac{1}{8}$ (cf. \eq{eq:lem8_5}).
By \cref{eq:lem8_8,eq:lem8_9}, we have
\begin{equation}
	\begin{aligned}
		J_{1}+J_{2}\leq C(r) x^{r} \int_{u>  (1+\delta)x}
						  \xi\biggl(\frac{u}{x}\biggr)	u^{d-1} \exp
						  \biggl\{-\frac{u^{2}}{2}\biggr\}du.
	\end{aligned}
	\label{eq:lem8_10}
\end{equation}
Combining \cref{eq:lem8_7,eq:lem8_10}, we obtain
\begin{equation}
	\begin{aligned}
		&\quad\int_{u>  (1+\delta)x}
						  \xi\biggl(\frac{u}{x}\biggr)	u^{r+d-1} \exp \biggl\{-\frac{u^{2}}{2}+
	C\frac{x u^{2}}{\sqrt{n}}\biggr\}du\\&\leq C(r) x^{r} \int_{u>  (1+\delta)x}
						  \xi\biggl(\frac{u}{x}\biggr)	u^{d-1} \exp
    \biggl\{-\frac{u^{2}}{2}\biggr\}du\\&\quad+(r+d-2)\Bigl(1-\frac{2Cx}{\sqrt{n}}\Bigr)^{-1}\int_{u>
						  (1+\delta)x}
		  	\xi\biggl(\frac{u}{x}\biggr) u^{r+d-3} \exp \biggl\{-\frac{u^{2}}{2}+
			C\frac{x u^{2}}{\sqrt{n}}\biggr\}du.
	\end{aligned}
	\label{eq:lem8_11}
\end{equation}
If $r-2\geq 2-d$, we can apply \cref{eq:lem8_11} to the last integration. Performing
this procedure $p$ times, where $p$ is the smallest integer that is greater than or equal to $\frac{r+d}{2}-1$, we have 
\begin{equation}
	\begin{aligned}
		&\quad\int_{u>  (1+\delta)x}
						  \xi\biggl(\frac{u}{x}\biggr)	u^{r+d-1} \exp \biggl\{-\frac{u^{2}}{2}+
    C\frac{x u^{2}}{\sqrt{n}}\biggr\}du\\&\leq C(r) \biggl( \sum^{p}_{i=0} x^{r-2i} \biggr) \int_{u>  (1+\delta)x}
						  \xi\biggl(\frac{u}{x}\biggr)	u^{d-1} \exp
						  \biggl\{-\frac{u^{2}}{2}\biggr\}du\\
                                         &\quad+ \Bigl(1-\frac{2Cx}{\sqrt{n}}\Bigr)^{-p-1}
                                         \Biggl(\prod_{i=1}^{p+1}(r+d-2i)\Biggr)  \int_{u>
						  (1+\delta)x}
		  	\xi\biggl(\frac{u}{x}\biggr) u^{r+d-3-2p} \exp \biggl\{-\frac{u^{2}}{2}+
			C\frac{x u^{2}}{\sqrt{n}}\biggr\}du.
	\end{aligned}
	\label{eq:lem8_12}
\end{equation}
Because the last term is $\leq 0$, we then have 
\begin{equation}
	\begin{aligned}
		&\quad\int_{u>  (1+\delta)x}
						  \xi\biggl(\frac{u}{x}\biggr)	u^{r+d-1} \exp \biggl\{-\frac{u^{2}}{2}+
    C\frac{x u^{2}}{\sqrt{n}}\biggr\}du\\&\leq C(r) \biggl( \sum^{p}_{i=0} x^{r-2i} \biggr) \int_{u>  (1+\delta)x}
						  \xi\biggl(\frac{u}{x}\biggr)	u^{d-1} \exp
						  \biggl\{-\frac{u^{2}}{2}\biggr\}du
                                       \\&\leq C(r) x^{r} \int_{u>  (1+\delta)x}
						  \xi\biggl(\frac{u}{x}\biggr)	u^{d-1} \exp
						  \biggl\{-\frac{u^{2}}{2}\biggr\}du,
	\end{aligned}
	\label{eq:lem8_13}
\end{equation}
where the last inequality follows from $x> 1$.
We can easily verify that 
\begin{equation}
    \begin{aligned}
        &\int_{x<u\leq   (1+\delta)x}
						  \xi\biggl(\frac{u}{x}\biggr)	u^{r+d-1} \exp \biggl\{-\frac{u^{2}}{2}+
                          C\frac{x u^{2}}{\sqrt{n}}\biggr\}du
                          \\\leq& C(r) x^{r} \int_{  x<u\leq (1+\delta)x}
						  \xi\biggl(\frac{u}{x}\biggr)	u^{d-1} \exp
						  \biggl\{-\frac{u^{2}}{2}\biggr\}du.
    \end{aligned}
    \label{eq:lem8_14}
\end{equation}
By \cref{eq:lem8_6,eq:lem8_13,eq:lem8_14}, we have
\begin{equation}
    \begin{aligned}
       \int_{\lvert D y\rvert> x}^{ } \lvert y\rvert^{r} 
                          \exp \biggl\{-\frac{\lvert y\rvert^{2}}{2}+
                          C\frac{x \lvert y\rvert^{2}}{\sqrt{n}}\biggr\}
                          dy&\leq C(r) x^{r} \int_{  u> x}
						  \xi\biggl(\frac{u}{x}\biggr)	u^{d-1} \exp
						  \biggl\{-\frac{u^{2}}{2}\biggr\}du\\
                          &=C(r) x^{r}  \int_{\lvert D y\rvert> x}^{ } 
                          \exp \biggl\{-\frac{\lvert y\rvert^{2}}{2}\biggr\}
                          dy\\
                          &= C(r) x^{r} \P(|DZ|> x).
    \end{aligned}
    \label{eq:lem8_15}
\end{equation}
This proves \cref{lem:8}.

\end{proof}

\begin{proof}[Proof of \cref{lem:7}]
In this proof, we denote by $C_d$ positive constants that depend only on $d$. They may differ in different expressions.
All of the chi-square random variables below are assumed to be independent.
Because $x> 1$, we can verify that
$\lambda_{p-1}\geq 1/2$. 
Because $\lambda_{i}$ decreases with respect to $i$, we have 
\begin{equation}
	\begin{aligned}
		\P\biggl( \sum_{i=1}^{s} \lambda_{i} \chi^{2}_{v_{i}}\geq x^{2}\biggr)\geq\P\biggl(
		\lambda_{p-1} \chi^2_{ \sum^{p-1}_{j=1} v_{j}}+\sum_{i=p}^{s} \lambda_{i} \chi^{2}_{v_{i}}\geq
	x^{2}\biggr).
	\end{aligned}
	\label{eq:lem7_2}
\end{equation}
For any positive integer $v\leq d$, from chi-square tail probabilities, there exists a positive constant $C_{d}$ depending only on
$d$ such that 
\begin{equation}
	\begin{aligned}
		\P\left( \chi^{2}_{v}\geq a^{2} \right)\geq C_{d} a^{v-2} e^{-\frac{a^{2}}{2}}
		\mbox{for all $a^{2}\geq\frac{1}{2^{d+1}}$}.
	\end{aligned}
	\label{eq:lem7_3}
\end{equation}
From the definition  of $r$ and \cref{eq:lem7_3}, we have 
\begin{equation}
	\begin{aligned}
	&\P(\lambda_{p-1} \chi^2_{\sum_{j=1}^{p-1} v_j}\geq a^2)	=\P\left( \lambda_{p-1} \chi^{2}_{ r}\geq a^{2} \right)\\
	 \geq&		C_{d} \left(\frac{a}{\lambda^{1/2}_{p-1}}\right)^{r-2} e^{-\frac{a^{2}}{2\lambda_{p-1}}}\geq
		C_{d}  a^{r-2} e^{-\frac{a^{2}}{2\lambda_{p-1}}}\mbox{ for
		all $a^{2}\geq \frac{1}{2^{s-p+2}}$},
	\end{aligned}
	\label{eq:lem7_4}
\end{equation}
where in the last inequality we used the fact that $1/2\leq \lambda_{p-1}\leq 1$.

If $p=s+1$, then \cref{eq:lem7_1} follows from \cref{eq:lem7_2}, \cref{eq:lem7_4} with $a=x$, and $\bigl(\frac{1}{\lambda_{p-1}}-1\bigr)x^2\leq 1$.

Suppose now that $p\leq s$. We let $Y=\lambda_{p-1} \chi^{2}_{ r}$ and for any positive integer $v$, let $f_{v}(\cdot)$ be the density of
$\chi^{2}_{v}$. Then, for $x^{2}/2^{s-p+1}\leq a^{2}\leq x^2$, we have
\begin{equation}
	\begin{aligned}
		\P\left( Y+\lambda_{p} \chi_{v_{p}}^{2} \geq a^{2} \right)\geq
		\frac{1}{\lambda_{p}} \int_{0}^{a^{2}/2} f_{v_{p}}\biggl(\frac{y}{\lambda_{p}}\biggr)
		\P\left( Y\geq a^2 - y\right) dy.
	\end{aligned}
	\label{eq:lem7_5}
\end{equation}
In the above integration, $y\in [0,a^{2}/2]$, and thus $a^{2}\geq a^{2}-y\geq a^{2}/2$. Furthermore,
because $x^{2}/2^{s-p+1}\leq a^{2}\leq x^2$, we have 
\begin{equation}
	\begin{aligned}
		\frac{1}{2^{s-p+2}}<\frac{x^2}{2^{s-p+2}}\leq a^{2}-y\leq x^{2},
	\end{aligned}
	\label{eq:lem7_6}
\end{equation}
and we can apply \cref{eq:lem7_4} to $\P(Y\geq a^2-y)$.
Plugging \cref{eq:lem7_4} into \cref{eq:lem7_5} yields 
\begin{equation}
	\begin{aligned}
		\P\left(Y+ \lambda_{p} \chi_{v_{p}}^{2} \geq a^{2} \right)&\geq 
		\frac{C_{d}}{\lambda_{p}} \int_{0}^{a^{2}/2}
		\biggl(\frac{y}{\lambda_{p}}\biggr)^{\frac{v_{p}}{2}-1} e^{-\frac{y}{2
		\lambda_{p}}} (a^{2}-y)^{\frac{r}{2}-1} e^{-\frac{a^{2}-y}{2 \lambda_{p-1}}}dy
																\\&\geq
																C_{d}\lambda_{p}^{-\frac{v_{p}}{2}} \int_{0}^{a^{2}/2}
		y^{\frac{v_{p}}{2}-1} e^{-\frac{y}{2
			}(\frac{1}{\lambda_{p}}-\frac{1}{\lambda_{p-1}})}  dy
			a^{r-2}e^{-\frac{a^{2}}{2 \lambda_{p-1}}}
		\\&\geq C_{d}\lambda_{p}^{-\frac{v_{p}}{2}} \int_{0}^{a^{2}/2}
		y^{\frac{v_{p}}{2}-1} e^{-\frac{y}{2
			}(\frac{1}{\lambda_{p}}-1)}  dy
			a^{r-2}e^{-\frac{a^{2}}{2 }} e^{-\frac{a^{2}}{2}(\frac{1}{\lambda_{p-1}}-1)}.
	\end{aligned}
	\label{eq:lem7_7}
\end{equation}
By a change of variable, $\frac{1}{\lambda_{p-1}}-1\leq \frac{1}{x^2}\leq \frac{1}{a^2}$ and
$\frac{1}{\lambda_p}-1>\frac{1}{x^2}\geq \frac{1}{a^2 2^{s-p+1}}\geq \frac{1}{a^{2} 2^{d}}$, we have \cref{eq:lem7_7} is greater than or
equal to
\begin{equation}
	\begin{aligned}
	&	C_{d}\lambda_{p}^{-\frac{v_{p}}{2}} \left(\frac{1}{\lambda_{p}}-1
		\right)^{-\frac{v_{p}}{2}} \int_{0}^{a^{2}(\frac{1}{\lambda_{p}}-1)/2}
		y^{\frac{v_{p}}{2}-1} e^{-\frac{y}{2
			}}  dy \,
			a^{r-2}e^{-\frac{a^{2}}{2 }} 
			\\&\geq C_{d}\lambda_{p}^{-\frac{v_{p}}{2}} \left(\frac{1}{\lambda_{p}}-1
		\right)^{-\frac{v_{p}}{2}} \int_{0}^{1/2^{d+1}}
		y^{\frac{v_{p}}{2}-1} e^{-\frac{y}{2
			}}  dy\,
			a^{r-2}e^{-\frac{a^{2}}{2 }}\\
			  &\geq C_{d} \left(1-\lambda_{p} \right)^{-\frac{v_{p}}{2}} a^{r-2}e^{-\frac{a^{2}}{2
			  }}.
	\end{aligned}
	\label{eq:lem7_8}
\end{equation}
Repeating procedures \cref{eq:lem7_4,eq:lem7_5,eq:lem7_6,eq:lem7_7,eq:lem7_8} $s-p+1$ times for the right-hand side of \cref{eq:lem7_2}, we have,
for $a^{2}\in [x^{2}/2,x^2]$,
\begin{equation}
	\begin{aligned}
		\P\biggl( \sum^{s}_{i=1} \lambda_{i} \chi_{v_{i}}^{2}\geq a^{2}\biggr)\geq C_{d} \biggl[\prod_{i=p}^{s}
		(1-\lambda_{i})^{-\frac{v_{i}}{2}}\biggr] a^{r-2} e^{-\frac{a^{2}}{2}}.
	\end{aligned}
	\label{eq:lem7_9}
\end{equation}
Taking $a=x$ yields the desired result.

\end{proof}

\appendix

\section{Appendix}\label{s6}

\begin{proof}[Proof of \eq{f3}]
The result for bounded $x$ follows immediately from \eq{f1} and \eq{f2}. In the following, we assume $x> 1$.
We use $\delta, \epsilon$, and $\varepsilon$ to denote unspecified positive constants, which do not depend on $n$ and $x$.
By \cite[Theorem 3]{VonBahr1967}, we have, for some positive constant $\delta>0$ and $1<x\leq \delta \sqrt{n}$,
\besn{\label{1}
\P(|W|>x)=&(2\pi)^{-d/2} \int_{u\in \Omega_0} \exp\left( n\sum_{v=3}^\infty \left(\frac{x}{\sqrt{n}}\right)^v Q_v(u) \right) dS\\
&\times \int_x^\infty e^{-y^2/2} y^{d-1} dy \left(1+O\left(\frac{x}{\sqrt{n}}\right)\right),
}
where $dS$ is the surface measure of $\Omega_0=\{u\in \mathbb{R}^d: |u|=1\}$
and for each $v\geq 3$, $Q_v: \mathbb{R}^d\to \mathbb{R}$ is a homogeneous polynomial of degree $v$ whose coefficients depend on the mixed cumulants up to order $v$ of $X_1$.
For example, $Q_3(u)=\frac{1}{6}\sum_{j,k,l=1}^d \E[X_{1j}X_{1k}X_{1l}] u_j u_k u_l$, where $j,k,l$ are the indices of vector components.
Moreover, $\sum_{v=3}^\infty Q_v(u)$ is convergent for $|u|\leq \epsilon$, where $\epsilon$ is a positive constant.

In the remainder of the proof, assume that $x\leq \epsilon n^{1/4}$, which can be achieved by choosing the positive constant $\varepsilon$ in the range of $1< x\leq \varepsilon n^{1/6}$ to be sufficiently small.
Because $Q_v$ is a polynomial of degree $v$, we have
\bes{
&n\sum_{v=8}^\infty \left(\frac{x}{\sqrt{n}}\right)^v Q_v(u)=n\sum_{v=8}^\infty Q_v\left(u\frac{x}{\sqrt{n}}\right)\\
=&n\sum_{v=8}^\infty Q_v\left(u\frac{x}{n^{1/4}}\right) \left(\frac{1}{n^{1/4}}\right)^v=\sum_{v=8}^\infty n^{1-v/4} Q_v\left(u\frac{x}{n^{1/4}}\right)=O\left(\frac{1}{n}\right),
}
where in the last step, we used the Dirichlet condition for the convergence of series and the fact that $\sum_{v=3}^\infty Q_v(u)$ is convergent for $|u|\leq \epsilon$.
Therefore,
\besn{\label{2}
&\exp\left( n\sum_{v=3}^\infty \left(\frac{x}{\sqrt{n}}\right)^v Q_v(u) \right)\\
=&\exp\left( n\sum_{v=3}^7 \left(\frac{x}{\sqrt{n}}\right)^v Q_v(u) +n\sum_{v=8}^\infty \left(\frac{x}{\sqrt{n}}\right)^v Q_v(u)\right)\\
=&\exp\left( n \left(\frac{x}{\sqrt{n}}\right)^3 Q_3(u) \right)\left(1+O\left(\frac{x^4}{n}\right)\right).
}
From \eq{1} and \eq{2}, we have, for $1< x\leq \min\{\delta \sqrt{n}, \epsilon n^{1/4}\}$,
\bes{
\P(|W|>x)=&(2\pi)^{-d/2} \int_{u\in \Omega_0} \exp\left( n \left(\frac{x}{\sqrt{n}}\right)^3 Q_3(u) \right) dS\\
&\times \int_x^\infty e^{-y^2/2} y^{d-1} dy \left(1+O\left(\frac{x}{\sqrt{n}}+\frac{x^4}{n}\right)\right).
}
Therefore, for $1< x\leq \varepsilon n^{1/6}$ for a sufficiently small $\varepsilon>0$,
\bes{
\P(|W|>x)=&(2\pi)^{-d/2} \int_{u\in \Omega_0} \left( 1+ \frac{x^3}{\sqrt{n}} Q_3(u) +O\left(\frac{x^6}{n}\right)  \right) dS\\
&\times \int_x^\infty e^{-y^2/2} y^{d-1} dy \left(1+O\left(\frac{x}{\sqrt{n}}+\frac{x^4}{n}\right)\right).
}
By symmetry, because $Q_3$ is a polynomial of degree 3,
\be{
\int_{u\in \Omega_0}\frac{x^3}{\sqrt{n}} Q_3(u)   dS=0.
}
This result, together with the fact that
\ben{\label{5}
(2\pi)^{-d/2}\int_{u\in \Omega_0} dS\int_x^\infty e^{-y^2/2} y^{d-1} dy=\P(|Z|>x),
}
proves \eq{f3}.
\end{proof}

\begin{proof}[Proof of \eq{f23}]
From \eq{1} and \eq{2}, for $1< x_n\leq \min\{\delta \sqrt{n}, \epsilon n^{1/4}\}$ for a sufficiently small constant $\delta>0$, we have
\besn{\label{3}
\P(|W|>x_n)=&(2\pi)^{-d/2} \int_{u\in \Omega_0} \exp\left( n \left(\frac{x_n}{\sqrt{n}}\right)^3 Q_3(u) \right)\left(1+O\left(\frac{x_n^4}{n}\right)\right) dS\\
&\times \int_{x_n}^\infty e^{-y^2/2} y^{d-1} dy \left(1+O\left(\frac{x_n}{\sqrt{n}}\right)\right).
}
For $x_n=cn^{1/6}\ll n^{1/4}$, from \eq{3}, we have
\ben{\label{6}
\P(|W|>x_n)=(2\pi)^{-d/2} \int_{u\in \Omega_0} \exp\left( n \left(\frac{x_n}{\sqrt{n}}\right)^3 Q_3(u) \right) dS \int_{x_n}^\infty e^{-y^2/2} y^{d-1} dy (1+o(1)).
}
Recall $Q_3(u)=\frac{1}{6}\sum_{j,k,l=1}^d \E[X_{1j}X_{1k}X_{1l}] u_j u_k u_l$.
If the mixed third cumulants of $X_1$ are not all zero, then $Q_3(u)$ is a non-zero function. Moreover, $Q_3(u)=-Q_3(-u)$, and thus $\int_{u\in \Omega_0}Q_3(u)dS=0$.
This implies
\be{
\int_{u\in \Omega_0} \exp\left( n \left(\frac{x_n}{\sqrt{n}}\right)^3 Q_3(u) \right) dS=\int_{u\in \Omega_0} \exp\left( c^3 Q_3(u) \right) dS>\int_{u\in \Omega_0}  dS,
}
which, together with \eq{6} and \eq{5}, proves \eq{f23}.

\end{proof}

\section*{Acknowledgments}

Fang X. was partially supported by Hong Kong RGC 
ECS 24301617,  
GRF 14302418 and 14305821, 
a CUHK direct grant, and a CUHK start-up grant. 
Shao Q.M. was partially supported by National Nature Science Foundation of China NSFC 12031005 and
Shenzhen Outstanding Talents Training Fund.





\end{document}